\documentclass[12pt,reqno]{amsart}
\usepackage{amsmath,amssymb,amsfonts,amsthm}
\usepackage[colorlinks=true,linkcolor=blue,citecolor=blue]{hyperref}
\usepackage[margin=1in]{geometry}
\usepackage{xcolor}
\usepackage{enumitem}
\usepackage{epsfig}
\usepackage{graphicx}
\usepackage[]{subfig}

\allowdisplaybreaks

\newcommand{\vt}{\vartheta}
\newcommand{\R}{\mathbb{R}}
\newcommand{\N}{\mathbb{N}}
\newcommand{\eps}{\varepsilon}
\newcommand{\var}{\varepsilon}
\newcommand{\na}{\nabla}
\newcommand{\pa}{\partial}
\newcommand{\X}{\mathfrak{X}}
\newcommand{\Y}{\mathfrak{Y}}
\newcommand{\Z}{\mathbf{Z}}
\newcommand{\F}{\mathcal{F}}
\newcommand{\mm}{\mathfrak{M}}
\newcommand{\cc}{\mathfrak{C}}
\renewcommand{\S}{\mathsf{B}}
\newcommand{\lam}{\lambda}
\newcommand{\td}{\tilde}

\newcommand{\bra}[1]{\left(#1\right)}
\newcommand{\sbra}[1]{\left[#1\right]}
\newcommand{\abra}[1]{\left\{#1\right\}}
\newcommand{\norm}[1]{\left\|#1\right\|}
\newcommand{\inner}[2]{\left\langle#1,#2\right\rangle}
\newcommand{\abs}[1]{\left|#1\right|}

\newtheorem{theorem}{Theorem}[section]
\newtheorem{lemma}{Lemma}[section]
\newtheorem{proposition}{Proposition}[section]
\newtheorem{remark}{Remark}[section]

\title[A chemotaxis model involved in the modeling of multiple sclerosis]{Global well-posedness and nonlinear stability of a chemotaxis system modeling multiple sclerosis}
\author{Laurent Desvillettes, Valeria Giunta, Jeff Morgan, Bao Quoc Tang}

\address{\small  Universit\'{e} de Paris, Sorbonne Universit\'e, CNRS\\
 Institut de Math\'{e}matiques de Jussieu-Paris Rive Gauche}
\email{desvillettes@math.univ-paris-diderot.fr}

\address{\small  Department of Engineering, University of Palermo, Italy}
\email{valeria.giunta@unipa.it}

\address{\small Department of Mathematics, University of Houston, Houston, Texas 77004, USA}
\email{jmorgan@math.uh.edu}

\address{\small Institute of Mathematics and Scientific Computing, University of Graz, Heinrichstrasse
	36, 8010 Graz, Austria}
\email{quoc.tang@uni-graz.at, baotangquoc@gmail.com}

\begin{document}
	\begin{abstract}
		We consider a system of reaction--diffusion equations including chemotaxis terms and coming out of the modeling of multiple sclerosis. 
		The global existence of strong solutions to this system in any dimension is proved, and it is also shown that the solution is bounded uniformly in time. Finally, a nonlinear stability result is obtained when the chemotaxis term is not too big. We also perform numerical simulations to show the appearance of Turing patterns when the chemotaxis term is large.
	\end{abstract}
	\maketitle
	
	\section{Introduction and Main results}
	
	In this paper, we study the global existence of strong solutions (and also some aspects of  the large time behavior) of the following reaction-diffusion system including chemotaxis terms:
	\begin{equation}\label{sys_original}
	\left\{
	\begin{aligned}
	\partial_t m &= \Delta m + m(1-m^{a-1}) - \chi \na\cdot(f(m)\na c), &&(x,t)\in Q_T,\\
	\pa_t c &= \eps_0\Delta c + \delta d - c + \beta m, &&(x,t)\in Q_T,\\
	\pa_t d &= g(m)(1-d), &&(x,t)\in Q_T,\\
	\na m\cdot \nu &=\na c\cdot \nu = 0, &&(x,t)\in \pa\Omega\times(0,T),\\
	m(x,0) &= m_0(x), c(x,0) = c_0(x), d(x,0) = d_0(x), &&x\in\Omega ,
	\end{aligned}
	\right.
	\end{equation}
	where $\Omega\subset \mathbb R^n$ is a bounded domain with smooth boundary $\partial\Omega$, $Q_T := \Omega \times (0,T)$, $\nu$ is the outward normal vector to $\partial\Omega$ at a point of $\pa\Omega$, the parameters satisfy $a>1$, and $\chi, \eps_0, \delta, \beta > 0$, and the initial data is nonnegative, bounded and sufficiently smooth. The nonlinearities $f$ and $g$ satisfy some conditions which will be specified later.

	\medskip	
	The system \eqref{sys_original} with 
	\begin{equation}\label{special}
		a = 2, \quad 
		f(m) = \frac{m}{1+m} \quad \text{ and } \quad g(m) = \frac{rm^2}{1+m}, \; r>0,
	\end{equation}
	was recently proposed \cite{lombardo2017demyelination} for the dynamics of multiple sclerosis. There, $m := m(x,t) \ge 0$ is the density of inflammatory immune cells (macrophages), $c := c(x,t) \ge 0$ is the density of a chemoattractant (cytokine), and $d := d(x,t) \ge 0$ is the density of destroyed oligodendrocytes. This model generalizes that of \cite{calvez2008mathematical} and \cite{khonsari2007origins} to describe a rare and aggressive form of multiple sclerosis, namely Bal\'o's sclerosis. It has been shown that by varying the parameter values within realistic ranges taken from the experimental literature, this model is
	able to reproduce different pathological scenarios typical of the disease (\cite{BGGLPS18}, \cite{BGGLPS19}). The analysis of \eqref{sys_original}--\eqref{special} has been carried out in two recent works: in \cite{desvillettes2020existence}, the global existence of very weak and classical solutions to \eqref{sys_original}--\eqref{special} was shown in dimension one; global classical solutions in higher dimensions are shown in \cite{hu2020global} under the smallness assumption of the chemoattractant coefficient $\chi$ (when $n>3$). This latter work also shows that the solution is bounded uniformly in time, and the positive equilibrium is globally asymptotically stable if $\chi$ is small enough.

	\medskip

	In the present work, we study \eqref{sys_original} with a general logistic growth of order $a>1$, and the nonlinearities $f$ and $g$ are assumed to satisfy
	\begin{enumerate}[label=(F),ref=F]
		\item\label{F} $f\in C^2([0,\infty))$ with $f(0) = 0$, and there exist constants $\gamma, b, \ell \geq 0$ such that
		\begin{equation}\label{assumption_f}
		|f(y)| \leq \gamma \, y^b \quad \text{ and } \quad |f'(y)|\leq \gamma\,  y^\ell \quad \text{ for all } \quad y\geq 0 ,
		\end{equation}
	\end{enumerate}
	\begin{enumerate}[label=(G),ref=G]
		\item\label{G} $g\in C^1([0,\infty))$ and $g(y) > 0$ for all $y > 0$.
	\end{enumerate}
	It's easy to check that the special case \eqref{special} fulfills the assumptions \eqref{F} for $f$ with $b = \ell =0$, and the assumption \eqref{G} for $g$.
	
	\medskip
	The first main result of this paper concerns the existence and uniqueness of (uniformly w.r.t. $t$) bounded strong solutions to system (\ref{sys_original}), when the parameter $b$ appearing in assumption 
    (\ref{assumption_f}) is not too large.

	\begin{theorem}\label{thm:main}
Let $\Omega$ be a smooth ($C^{2+\alpha}$ for some $\alpha>0$) bounded connected open subset of $\R^n$ (for $n \in \N - \{0\}$), and suppose $a>1$, $\chi, \eps_0, \delta, \beta > 0$, $\gamma, b, l \ge 0$, and the functions $f$ and $g$ satisfy (F) and (G).
We further assume 
		\begin{equation}\label{e2}
		\max\abra{b-1,\frac b2} < \frac{a}{n+2},
		\end{equation}
	and consider  nonnegative initial data $(m_0, c_0, d_0)\in W^{(2-2/\tilde{p}),\tilde p}(\Omega)\times W^{(2-2/\tilde{p}),\tilde p}(\Omega)\times L^\infty(\Omega)$, for some $\tilde{p}>\max\{n+2,a\}$.
Then,  system \eqref{sys_original} has a unique nonnegative (for each component) strong solution which is bounded uniformly in time. More precisely, $\pa_t m$, $\pa_t c$, $\pa_t d$, $\Delta m$, $\Delta c$, $ \na\cdot(f(m)\na c)$ belong to $L^{\tilde p}(\Omega\times(0,T))$ for all $T>0$, and there exists $C>0$ (depending on $\Omega$, $n$, $a$, $b$, $\var_0$, $\beta$, $\delta$, $m_0$, $c_0$, $d_0$), and $\tilde C>0$ (depending on $\Omega$, $n$, $a$, $b$, $\var_0$, $\beta$, $\delta$, $d_0$), such that
		\begin{equation}\label{uniform-in-time}
			\sup_{t\geq 0}\bra{\|m(\cdot , t)\|_{L^\infty(\Omega)}+\|c(\cdot , t)\|_{W^{1,\infty}(\Omega)}+\|d(\cdot , t)\|_{L^\infty(\Omega)}} \leq C,
		\end{equation}
		and
         \begin{equation}\label{eventual-uniform-in-time}
			\limsup_{t\to \infty}\bra{\|m(\cdot , t)\|_{L^{\infty}(\Omega)}+\|c(\cdot , t)\|_{W^{1,\infty}(\Omega)}+\|d(\cdot , t)\|_{L^\infty(\Omega)}} \leq \tilde C.
		\end{equation}
		\medskip
		Finally, if the initial data are smooth, i.e. $m_0, c_0, d_0 \in C^2(\overline{\Omega})$ and satisfy compatibility conditions $\na m_0 \cdot \nu = \na c_0 \cdot \nu = 0$ on $\pa\Omega$, then the solution is classical. That is $$\pa_tm, \pa_tc, \pa_td, \Delta m, \Delta c, \na\cdot(f(m)\na c) \in C^0(\overline{\Omega}\times [0,T])$$ for all $T>0$.
	\end{theorem}
	\begin{remark}
		The difference between $C$ and $\tilde{C}$ in \eqref{uniform-in-time} and \eqref{eventual-uniform-in-time} is that $\tilde{C}$ does not depend on $m_0$ and $c_0$. This means that in large time, the bound of solutions depends only on the size of $d_0$ but not $m_0$ and $c_0$.
	\end{remark}
	For the special case \eqref{special}, since $b = 0$ and $a = 2$, condition \eqref{e2} is satisfied in all dimensions. Therefore, Theorem \ref{thm:main} in particular shows that the system \eqref{sys_original}--\eqref{special} has global unique strong solutions (for suitable initial data) in all dimensions. We also emphasize that Theorem \ref{thm:main} improves the global existence in \cite{hu2020global} by removing the requirement that $\chi$ is small in all cases.
	
	
	Let us briefly describe the ideas underlying the proof  of Theorem \ref{thm:main}. Using the comparison principle, it's easy to show that $d$ is bounded uniformly in time and space. It follows straightforwardly from the logistic term that $m$ is bounded in $L^a(\Omega\times(0,T))$. Using these bounds and the properties of the heat equation satisfied by $c$, thanks to its linearity, we obtain some $L^q(\Omega\times(0,T))$ bound for $\na c$. In order to use this estimate in the equation satisfied by $m$, we exploit the so-called {\it duality method}, which proved its significant usefulness in the study of reaction-diffusion systems (cf. \cite{canizo2014improved,morgan2020boundedness,pierre10}), or cross diffusion systems (\cite{DTres, DLMT}). Under condition \eqref{e2}, the duality method yields an $L^p(\Omega\times(0,T))$ estimate for $m$ for all $1\leq p < \infty$. This information and a bootstrap argument allow us to conclude that $c$ is bounded in $L^\infty(0,T;W^{1,\infty}(\Omega))$ and that $m$ is bounded in $L^\infty(\Omega\times(0,T))$, hence the global existence of bounded solutions to \eqref{sys_original}. To show the uniform-in-time bound \eqref{uniform-in-time}, we use a smooth truncation function in time to study \eqref{sys_original} on each cylinder $\Omega\times(\tau,\tau+1)$, $\tau\in \mathbb N$. We repeat the previous arguments to obtain that $m$ and $c$ are bounded in $\Omega\times(\tau,\tau+1)$ {\it uniformly in $\tau\in \mathbb N$}, which consequently implies the desired bounds \eqref{uniform-in-time} and \eqref{eventual-uniform-in-time}. It's worth 
noting that this uniform-in-time bound plays an important role in the nonlinear stability, which is discussed in the following paragraphs.

	\medskip
	Our second main result of this paper is the nonlinear exponential stability of a positive (for all components) homogeneous equilibrium. A homogeneous equilibrium $(m^*,c^*,d^*)$ to \eqref{sys_original} solves $m^* = (m^*)^a$, $\delta d^* + \beta m^* = c^*$ and $g(m^*)(1-d^*) = 0$. It's straightforward that $$(\bar m, \bar c, \bar d) = (1, \beta +\delta,1)$$ is the unique {\it positive} (for all components) homogeneous equilibrium. Besides this, we also have other equilibria, namely:
	\begin{equation*}
		(m^*,c^*,d^*) = \begin{cases}
			(0,\delta,1) &\text{ if } g(0)\ne 0,\\
			(0,\delta\zeta,\zeta) \text{ with } \zeta \geq 0 \text{ arbitrary }  &\text{ if } g(0) = 0.
		\end{cases}
	\end{equation*}
	
	\begin{theorem}
		[Exponential nonlinear stability of $(\bar{m}, \bar{c}, \bar{d})$]\label{thm:main2}
		Let the assumptions of Theorem~\ref{thm:main} hold and assume additionally that $f(1)>0$. If
		\begin{equation}\label{sub_critical}
			\chi < \chi_{\mathrm{subcrit}} := \frac{4\sqrt{\eps_0(a-1)}}{\beta f(1)},
		\end{equation}
		then the constant steady state $(\bar{m}, \bar{c}, \bar{d})$ is exponentially asymptotically stable. More precisely, there exists $\eps>0$ such that if $(m_0,c_0,d_0)\in W^{(2-2/\tilde{p}),\tilde{p}}(\Omega)\times W^{(2-2/\tilde{p}),\tilde{p}}(\Omega)\times L^\infty(\Omega)$ for $\tilde{p}>\max(n+2,a)$, and
		\begin{equation*}
			\|m_0 - \bar{m}\|_{L^2(\Omega)} + \|c_0 - \bar{c}\|_{L^2(\Omega)} + \|d_0 - \bar{d}\|_{L^2(\Omega)} \leq \eps,
		\end{equation*}
		then there are $C, \varrho > 0$ such that
		\begin{equation*}
			\norm{m(t) - \bar{m}}_{L^\infty(\Omega)} + \norm{c(t) - \bar{c}}_{W^{1,\infty}(\Omega)} + \norm{d(t) - \bar{d}}_{L^\infty(\Omega)} \leq C \,e^{-\varrho  t} \quad \text{ for all } t\geq 0.
		\end{equation*}
	\end{theorem}
	\begin{remark}
		We expect that the nonlinear stability in Theorem \ref{thm:main2} can be shown under the optimal condition $\chi < \chi_c$, where $\chi_c$ is defined in \eqref{chi_c} using the eigenvalues of the Laplacian with Neumann boundary condition (see section \ref{sec:stability} for more details). 
 We remark that the threshold \eqref{sub_critical} becomes optimal when $\eps_0^{-1} = a-1$, which makes $\chi_{\mathrm{subcrit}} = \chi_{c0}$ (see \eqref{chi_c0}).
	\end{remark}
	To prove Theorem \ref{thm:main2}, we first rewrite system \eqref{sys_original} with a new variable $(\tilde{m}, \tilde{c}, \tilde{d}) = (m - \bar m, c - \bar c, d - \bar d)$. Next, we show that under condition \eqref{sub_critical}, the linear part, after a suitable scaling, has a spectral gap. This spectral gap and the uniform-in-time bound proved in Theorem \ref{thm:main} allow us to show that the nonlinear part is dominated by the linear one, and consequently obtain the nonlinear stability of $(\bar m, \bar c, \bar d)$.
	
	
	\medskip
	{\bf The organization of the paper is as follows}: In  section \ref{sec2}, we prove Theorem~\ref{thm:main} by first showing global existence, 
 then uniqueness, and finally the uniform-in-time bound. In subsection \ref{local_stabil}, we prove the nonlinear stability in Theorem \ref{thm:main2}. Then, in subsection \ref{subsec:Turing} we perform numerical simulations on a 2D spatial domain to show Turing patterns which appear for sufficiently large chemotractant coefficient $\chi$. Finally, subsection \ref{subsec:GlobalStab} concludes the paper with a presentation of open problems relevant for the 
model under study.
	
	\medskip
	{\bf Notation.} Throughout this paper, we use the following notations
	\begin{itemize}
         \item The space $L^{p}(\Omega)$, $1\leq p\leq \infty$, is equipped with the norm
		\begin{equation*}
			\|v\|_{L^p(\Omega)} = \bra{\int_{\Omega}|v(x)|^pdx}^{\frac 1p} \quad \text{ when } \quad p<\infty,
		\end{equation*}
		and
		\begin{equation*}
			\|v\|_{L^\infty(\Omega)} = \underset{x\in\Omega}{\mathrm{ess\, sup}}|v(x)|.
		\end{equation*}
         \item When $1\leq p\leq \infty$ and $r>0$ is an integer, we define the space
              $$W^{r,p}(\Omega) := \abra{v\in L^p(\Omega): \pa_x^sv \in L^p(Q_{\tau,T}) \text{ for } s\in \N \text{ with } s \leq r},$$
         and when $r$ is not an integer, we use the definition in \cite[Chapter 2, Section 2]{ladyvzenskaja1988linear}.
		\item For any $0\leq\tau<T$, $Q_{\tau,T} := \Omega\times(\tau,T)$. When $\tau = 0$, we write $Q_T$ instead of $Q_{0,T}$.
		\item The space-time  space $L^{p}(Q_{\tau,T}) = L^p((\tau,T);L^p(\Omega))$, $1\leq p\leq \infty$, is equipped with the norm
		\begin{equation*}
			\|v\|_{L^p(Q_{\tau,T})} = \bra{\int_{\tau}^T\int_{\Omega}|v(x,t)|^pdxdt}^{\frac 1p} \quad \text{ when } \quad p<\infty,
		\end{equation*}
		and
		\begin{equation*}
			\|v\|_{L^\infty(Q_{\tau,T})} = \underset{(x,t)\in Q_{\tau,T}}{\mathrm{ess\, sup}}|v(x,t)|.
		\end{equation*}
		\item We denote the space-time space
		$$ W^{2,1}_{p}(Q_{\tau,T}) := \abra{v\in L^p(Q_{\tau,T}): \pa_t^r\pa_x^sv \in L^p(Q_{\tau,T}) \text{ for } r,s\in \N \text{ with } 2r + s \leq 2}, $$
		equipped with the norm
		$$ \|v\|_{W^{2,1}_{p}(Q_{\tau,T})} := \sum_{2r+s\leq 2}\norm{\pa_t^r\pa_x^s v}_{L^p(Q_{\tau,T})}.$$
	\end{itemize}

	\section{Proof of Theorem \ref{thm:main}} \label{sec2}
	\subsection{Global existence of bounded solutions}\label{sec:global} 
\ 
\medskip

In this subsection, we show the existence of solutions to system (\ref{sys_original}) which are bounded on  $Q_T$ for all $T>0$. 

The following lemma is useful in our analysis.
\begin{lemma}(\cite[Lemma 3.3]{ladyvzenskaja1988linear} and \cite[Theorem 3]{simon1986compact})\label{embedding}
	Assume $1<p<\infty$. There exists a constant $C>0$, only depending on
$T-\tau,\Omega,p,n$, such that for all $v\in W_p^{2,1}(Q_{\tau,T})$,
	\begin{itemize}
		\item[(a)]
		\begin{equation*}
			\norm{v}_{L^q(Q_{\tau,T})} \leq C\, \norm{v}_{W^{2,1}_p(Q_{\tau,T})}
		\end{equation*}
		where
		\begin{equation*}
			q = \begin{cases}
				\frac{(n+2)p}{n+2-2p} &\text{ if } p<(n+2)/2,\\
				<+\infty \text{ arbitrary } &\text{ if } p = (n+2)/2,\\
				+\infty &\text{ if } p > (n+2)/2.
			\end{cases}
		\end{equation*}
		\item[(b)]
		\begin{equation*}
		\norm{\na v}_{L^q(Q_{\tau,T})} \leq C\, \norm{v}_{W^{2,1}_p(Q_{\tau,T})}
		\end{equation*}
		where
		\begin{equation*}
		q = \begin{cases}
		\frac{(n+2)p}{n+2-p} &\text{ if } p<n+2,\\
		<+\infty \text{ arbitrary }  &\text{ if } p = n+2,\\
		+\infty &\text{ if } p > n+2.
		\end{cases}
		\end{equation*}
	\end{itemize}
Furthermore, when $p>n+2$, the space $W_p^{2,1}(Q_{\tau,T})$ is compactly embedded in $C^{1,0}(\overline\Omega\times [0,T])$.
\end{lemma}

Next, we prove the
\begin{proposition}\label{prop11}
Under the assumptions of Theorem \ref{thm:main}, and picking $T>0$,
system \eqref{sys_original} has a nonnegative (for each component) strong solution which is bounded on $Q_T$. More precisely,  there exists $C_T>0$  (depending on $\Omega$, $T$, $n$, $a$, $b$, $\var_0$, $\beta$, $\delta$, $\|m_0\|_{L^{\infty}(\Omega)}$, $\|c_0\|_{W^{ 1,\infty}(\Omega)}$, $\|d_0\|_{L^{\infty}(\Omega)}$), such that
		\begin{equation} \label{pourun}
			\|m\|_{L^\infty(Q_T)}+\|c\|_{L^{\infty}((0,T); W^{1,\infty}(\Omega))}+\|d\|_{L^\infty(Q_T)} \leq C_T.
		\end{equation}
\end{proposition}
\medskip

 {\it{Proof of Proposition \ref{prop11} }}:
	For a given $r>0$, define $h_r \in C^\infty(\R_+,[0,1])$ such that $h_r(y) = 1$ for $0\leq y\leq r$, and $h_r(y) = 0$ for $y\geq 2r$. We also use the standard convention $m_+ := \max\{0; m\}$. Consider the system
	\begin{equation}\label{sys_r}
	\left\{
	\begin{aligned}
		\pa_t m &=\Delta m + m_+(1-m_+^{a-1}) - \chi\na\cdot(f_r(m_+)\na c), &&(x,t)\in Q_T,\\
		\pa_t c &= \eps_0 \Delta c + \delta d - c + \beta m_+, &&(x,t)\in Q_T,\\
		\pa_t d &= g(m_+)(1-d), &&(x,t)\in Q_T,\\
		\na m\cdot \nu &= \na c\cdot \nu = 0, && (x,t)\in \pa\Omega\times(0,T),\\
		m(x,0) &= m_0(x), \,\, c(x,0) = c_0(x), \,\, d(x,0) = d_0(x), &&x\in\Omega,
	\end{aligned}
	\right.
	\end{equation}
	where the function $f_r$ is defined as
	\begin{equation*}
		f_r(y) := h_r(f(y)).
	\end{equation*}
	
\medskip

	We will prove, thanks to Leray-Schauder's theorem (cf. \cite[Theorem 10.3]{gilbarg2015elliptic}), that \eqref{sys_r} has a nonnegative (for each component) solution which is bounded in sup-norm {\it independent of $r>0$}, and consequently prove the existence of a solution to the original system \eqref{sys_original} on $Q_T$.
	To this end, we define
	\begin{equation*}
		\Y:= C^{1,0}(\overline\Omega\times [0,T] ) \quad \text{ and } \X := \Y\times \Y,
	\end{equation*}
where $C^{1,0}$ is the space of continuous functions such that their gradient (with respect to space) is also continuous,
	associated with the natural norm (making it a Banach space)
	\begin{equation*}
		\|u\|_{\Y} = \sup_{t\in[0,T]}\|u(t)\|_{L^\infty(\Omega)} + \sup_{t\in [0,T]}\|\na u(t)\|_{L^\infty(\Omega)},
	\end{equation*}
	and 
	\begin{equation*}
		\F_r: \X \times [0,1] \to \X \quad \text{ via } \quad \F_r(m,c, \lambda):= (\lambda\,\mm, \lambda\, \cc),
	\end{equation*}
	where $(\mm,\cc)$ is given (in a unique way) by solving first an ODE (for $d$), and then successively (for $\cc$ and $\mm$, in that order) 
two Neumann problems for the heat equation with a source. More precisely, $(\mm,\cc)$ is the (unique) solution on $Q_T$ of the system
	\begin{equation}\label{e6_r}
	\left\{
	\begin{aligned}
	\pa_t	\mm &= \Delta \mm + m_+\bra{1-m_+^{a-1}} - \chi \na\cdot(f_r(m_+)\na \cc), &&(x,t)\in Q_T,\\
	\pa_t	\cc &= \eps_0 \Delta \cc + \delta d - \cc + \beta m_+, &&(x,t)\in Q_T,\\
	\pa_t	d &= g(m_+)(1-d), &&(x,t)\in Q_T,\\
		\na\mm \cdot \nu &= \na \cc \cdot \nu = 0, &&(x,t)\in \pa\Omega\times(0,T),\\
		\mm(x,0) &= m_0(x), \cc(x,0) = c_0(x), d(x,0) = d_0(x), &&x\in\Omega.
	\end{aligned}
	\right.
	\end{equation}
Note first that $ \F_r(m,c, 0) = 0$.
Then,  regardless of $m$, we always have that
	\begin{equation}\label{e7}
		0 \leq d(x,t) \leq  \mu := \max\abra{1, \|d_0\|_{L^\infty(\Omega)}}. 
	\end{equation}

	Also for any $1<p\le \tilde p$, there exists $C_p>0$ depending on $\|m\|_{\Y}$ and $r$ (and $\Omega, T, m_0, c_0, \mu$) such that, thanks to maximal regularity results,
	\begin{equation}\label{e8}
		\|\cc\|_{W^{2,1}_{p}(Q_T)} \leq C_p,
	\end{equation}
	and
	\begin{equation}\label{e9}
		\|\mm\|_{W^{2,1}_{p}(Q_T)} \leq C_p.
	\end{equation}
	Thanks to the embedding in Lemma \ref{embedding},
	\begin{equation}\label{em}
		W^{2,1}_p(Q_T) \hookrightarrow \Y \text{ compactly when } p >n+2.
	\end{equation}
	 So, we see that since $\tilde p>n+2$, the map  $\F_r$ sends bounded sets in $\X$ into compact sets of $\X$. We now show that for any $\lambda\in [0,1]$, $\F_r(\cdot,\cdot,\lambda)$ is continuous from $\X$ to $\X$. Let $(m_1, c_1), (m_2,c_2)\in \X$ and $(\lambda\mm_1, \lambda\cc_1)= \F_r(m_1,c_1,\lambda)$, $(\lambda\mm_2,\lambda\cc_2) = \F_r(m_2,c_2,\lambda)$. 
We have
	\begin{equation}\label{difference}
	\begin{cases}
		\pa_t(\mm_1 - \mm_2) = \Delta(\mm_1 - \mm_2) + \sbra{m_{1+}(1-m_{1+}^{a-1}) - m_{2+}(1-m_{2+}^{a-1}) }\\
		\hspace{5.45cm} + \chi\na\cdot\sbra{f_r(m_{1+})\na \cc_1 - f_r(m_{2+})\na \cc_2},\\
		\pa_t(\cc_1 - \cc_2) = \eps_0\Delta (\cc_1 - \cc_2) + \delta(d_1 - d_2) - (\cc_1 - \cc_2) + \beta(m_{1+} - m_{2+}),\\
		(d_1 - d_2)_t = g(m_{1+})(1-d_1) - g(m_{2+})(1-d_2).
	\end{cases}
	\end{equation}
	By rewriting $g(m_{1+})(1-d_1) - g(m_{2+})(1-d_2) = g(m_{1+})(d_2 - d_1) + (g(m_{1+}) - g(m_{2+}))\, (1-d_2)$, and using $g\in C^1([0,\infty))$, we obtain 
	\begin{equation*}
	\begin{aligned}
		d_1(x,t) - d_2(x,t) =  \int_0^te^{-\int_{s}^tg(m_{1+}(r))dr}g'(\theta(x, s))(m_{1+}(s)-m_{2+}(s))(1-d_2(s))\,ds,
	\end{aligned}
	\end{equation*}
	where $\theta(x,s)$ is between $m_{1+}(x,s)$ and $m_{2+}(x,s)$. Therefore, there exists $C_T>0$ (which can depend on the norm $\|m_{1}\|_{\Y}$) such that
	\begin{equation*}
		\|d_1 - d_2\|_{L^\infty(Q_T)} \leq C_{T}\|m_{1+} - m_{2+}\|_{L^\infty(Q_T)} \leq C_T\|m_{1}-m_{2}\|_{\Y}.
	\end{equation*}
	By maximal regularity results and \eqref{em}
	\begin{equation}\label{estimate_C}
	\begin{aligned}
		\|\cc_1 - \cc_2\|_{\Y} \leq C_T\|\cc_1 - \cc_2\|_{W^{2,1}_p(Q_T)} & \leq C_T\bra{\|d_1-d_2\|_{L^p(Q_T)} + \|m_{1+} - m_{2+}\|_{L^p(Q_T}}\\
		& \leq C_T\|m_{1} - m_{2}\|_{\Y}.
	\end{aligned}
	\end{equation}
	For the equation satisfied by $\mm_1 - \mm_2$, we write 
	\begin{equation*}
		\abs{m_{1+}(1-m_{1+}^{a-1}) - m_{2+}(1-m_{2+}^{a-1})} \le \abs{m_{1} - m_{2}} 
+ \abs{m_{1} - m_{2}}\, \max(m_{1+}^{a-1}, m_{2+}^{a-1}),
	\end{equation*}
	and
	\begin{equation*}
	\begin{aligned}
		&\abs{\na\cdot\sbra{f_r(m_{1+})\na \cc_1 - f_r(m_{2+})\na \cc_2}}\\
		&\le \abs{f_r(m_{1+})}\abs{\Delta(\cc_1 - \cc_2)} + \abs{f_r(m_{1+}) - f_r(m_{2+})}\abs{\Delta \cc_2} + \abs{f_r'(m_{1+})}\abs{\na m_{1+}}\abs{\na(\cc_1-\cc_2)}\\
		&\quad + \abs{f_r'(m_{1+})}\abs{\na(m_{1+}-m_{2+})}\abs{\na \cc_2} + \abs{f_r'(m_{1+}) - f_r'(m_{2+})}\abs{\na m_{2+}}\abs{\na \cc_2} .
	\end{aligned}
	\end{equation*}
	Using the fact that $f\in C^2([0,\infty))$,  estimate \eqref{estimate_C}, and maximal regularity results we get
	\begin{equation}\label{estimate_M}
	\|\mm_1 - \mm_2\|_{\Y} \leq C_T\|\mm_1 - \mm_2\|_{W^{2,1}_{\tilde p}(Q_T)} \leq C_T\|m_{1}-m_{2}\|_{\Y}.
	\end{equation}
	It follows from \eqref{estimate_C} and \eqref{estimate_M} that $\F_r(\cdot,\cdot,\lam)$ is continuous from $\X$ to $\X$ for any $\lam\in [0,1]$.
\medskip

We now check the last assumption in Leray-Schauder's theorem.  We consider therefore
	\begin{equation}\label{def_Z}
	\Z:= \abra{(m,c)\in \X \,:\, (m,c) = \F_r(m,c, \lambda) \; \text{ where } \; 0<\lambda \leq 1 },
	\end{equation}
and will show that $\Z$ is bounded in $\X$. 
\medskip

	Note that if $(m,c)\in \Z$, then $(m,c) = (\lambda \mm, \lambda \cc)$, where $(\mm,\cc)$ solves \eqref{e6_r}. 
	Therefore, by multiplying the equations in \eqref{e6_r} by $\lambda$, we obtain (for $(m,c) \in \Z$)
	\begin{equation}\label{e10_r}
	\left\{
	\begin{aligned}
		\pa_t m &= \Delta m + \lambda\sbra{m_+(1-m_+^{a-1}) - \chi\na\cdot(f_r(m_+)\na \cc)}, &&(x,t)\in Q_T,\\
		\pa_t c &= \eps_0\Delta c - c + \lambda(\delta d + \beta m_+),&&(x,t)\in Q_T,\\
		\pa_t d &= g(m_+)(1-d), &&(x,t)\in Q_T,\\
		\na m\cdot \nu &=\na c\cdot \nu = 0, &&(x,t)\in \pa\Omega\times(0,T),\\
		m(x,0) &= \lambda m_0(x), c(x,0) = \lambda c_0(x), d(x,0) = d_0(x), &&x\in\Omega.
	\end{aligned}
	\right.
	\end{equation}
	Since $d\geq 0$, $m_+\geq 0$, and $c_0 \geq 0$, we obtain immediately  (for $(m,c) \in \Z$) 
that $c\geq 0$. We now show that $m$ is also nonnegative. Indeed, denote by $m_- = \max\{0,-m\}$, and multiply the equation 
satisfied by $m$ 
in \eqref{e10_r} by $m_-^2$. Then, an integration by parts gives for all $t \in [0,T]$,
	\begin{equation*}
		-\frac 13\int_{\Omega}m_-^3(x,t)dx - 2\int_0^t\int_{\Omega}m_-|\na m_-|^2dxdt \geq 0.
	\end{equation*}
	Therefore, $m_- = 0$ and thus $m\geq 0$. It follows from the nonnegativity of $m$, $c$ and $d$ (and system \eqref{e10_r}) that $(m,c,d)$ solves the system
	\begin{equation}\label{e11_r}
	\left\{
	\begin{aligned}
	\pa_t m &= \Delta m + \lambda m(1-m^{a-1}) - \chi\na\cdot(f_r(m)\na c), &&(x,t)\in Q_T,\\
	\pa_t c &= \eps_0\Delta c - c + \lambda(\delta d + \beta m),&&(x,t)\in Q_T,\\
	\pa_t d &= g(m)(1-d), &&(x,t)\in Q_T,\\
	\na m\cdot \nu &=\na c\cdot \nu = 0, &&(x,t)\in \pa\Omega\times(0,T),\\
	m(x,0) &= \lambda m_0(x), c(x,0) = \lambda c_0(x), d(x,0) = d_0(x), &&x\in\Omega.
	\end{aligned}
	\right.
	\end{equation}

In order to show that $\Z$ is bounded, we propose a series of lemmas. We start with the

	\begin{lemma}\label{l21}
 Under the assumptions of Theorem \ref{thm:main}, and supposing that $\lambda \in ]0,1]$, we consider
	 $(m,c,d)$ a strong, nonnegative (for each component) solution to \eqref{e11_r}. Then there exists $K_1>0$ (depending on $T$, $a, \beta, \delta$ and $\|m_0\|_{L^1(\Omega)}$, $\|c_0\|_{L^1(\Omega)}$, $\|d_0\|_{L^{\infty}(\Omega)}$, but not depending on $\lambda$) such that
		\begin{equation} \label{nou13}
			\sup_{t\in[0,T]}\bra{\|m(\cdot, t)\|_{L^1(\Omega)} + \|c(\cdot, t)\|_{L^1(\Omega)}} \leq K_1,
		\end{equation}
		and if $0\leq\tau<T$ then
		\begin{equation*}
			\|m\|_{L^a(Q_{\tau,T})} \leq \sbra{K_1(T-\tau+1)}^{1/a}.
		\end{equation*}
	\end{lemma}
	\begin{proof}
	Note that
	\begin{equation}\label{e12}
	\frac{d}{dt}\int_{\Omega} (m+c)dx = \int_{\Omega}\sbra{\lambda(\delta d + (\beta+1)m - m^a) - c}dx.
	\end{equation}
	Since $a>1$, there exists $k_a>0$ such that
	\begin{equation*}
		(\beta+1)m - m^a \leq k_a - m.
	\end{equation*}
	Applying this and the fact that $0\leq d \leq \mu$ to \eqref{e12},
 we obtain some $K_1>0$
 such that (\ref{nou13}) holds.
	 In addition, integrating the equation satisfied by $m$ in \eqref{e11_r} gives
	 \begin{equation*}
		 \int_{\Omega}m(x,T)dx = \int_\tau^T\int_{\Omega}\lambda m(1-m^{a-1})dxdt + \lambda\int_{\Omega}m(x,\tau)dx.
	 \end{equation*}
	 Consequently, 
	 \begin{equation}\label{e15}
	 	\|m\|_{L^a(Q_{\tau,T})}^a \leq K_1(T-\tau+1).	 
	 \end{equation}
	 \end{proof}

We now turn to the

 	\begin{lemma}\label{l22}
 Under the assumptions of Theorem \ref{thm:main}, and supposing that $\lambda \in ]0,1]$, we consider
	 $(m,c,d)$ a strong, nonnegative (for each component) solution to \eqref{e11_r}. Then there exists $K_2>0$ (depending on $\Omega$, $T$,  $n$, $a$, $\var_0, \beta, \delta$ and $\|c_0\|_{W^{(2-2/a),a}(\Omega)}$, $\|d_0\|_{L^{\infty}(\Omega)}$, but not depending on $\lambda$) such that
	    \begin{equation}\label{e16}
	 	\|c\|_{W^{2,1}_{a}(Q_T)} + \|\na c\|_{L^{q_1}(Q_T)} \leq K_2\,K_1\,(T+1),
	 	\end{equation}
	 	where 
	 	\begin{equation}\label{q1}
	 	q_1 = \left\{
	 	\begin{aligned}&\frac{(n+2)a}{n+2-a}  &&\text{ if } a < n+2,\\
	 	&<+\infty \text{ arbitrary} &&\text{ if } a =  n+2, \\
         & = +\infty && \text{ if } a > n+2
	 	\end{aligned}\right.
	 	\end{equation}	
	\end{lemma}

	\begin{proof}
	 Returning to the equation of $c$ in \eqref{e11_r}, and applying Lemma \ref{embedding}, 
we get \eqref{e16}. 
	 \end{proof}

We finally prove the

 	\begin{lemma}\label{lem:m_Lp}
 Under the assumptions of Theorem \ref{thm:main}, and supposing that $\lambda \in ]0,1]$, we consider
	 $(m,c,d)$ a  strong, nonnegative (for each component) solution to \eqref{e11_r}. 
 Then for any $1 \le p<\infty$, there exists $F_p^*>0$ (depending on $\Omega$, $T$, $n$, $p$, $a$, $b$, $\var_0$, $\beta$, $\delta$, $\|m_0\|_{W^{(2-2/\tilde{p}),\tilde p}(\Omega)}$, $\|c_0\|_{W^{(2-2/\tilde{p}),\tilde p}(\Omega)}$, $\|d_0\|_{L^{\infty}(\Omega)}$,  but not depending on $\lambda$) such that
 		\begin{equation} \label{estimp}
	 		\|m\|_{L^p(Q_T)} \leq F_p^*.
 		\end{equation}
 	\end{lemma}

 	\begin{proof}
 		Note that from \eqref{e2}, it follows that there exists $s\in (0,1)$ (depending on $a$, $b$, $n$ only, and with $s$ close to $1$ if the inequality in \eqref{e2} is close to being an equality) satisfying
 	\begin{equation}\label{e3}
 	b-s < \frac{a}{n+2}, \quad \text{ and } \quad  
 	b-s<\frac{2a}{n+2} - 1.
 	\end{equation}
	It follows that
	 \begin{equation}\label{e18}
	 	b < s - 1 + \frac{2a}{n+2}
	 \end{equation}
	 and
	 \begin{equation}\label{e19}
	 	(n+2)(b-s)<a.
	 \end{equation}
	 Note also that 
	 \begin{equation*} \label{ut}
		 1 < \frac{2}{n+2} - \frac{b-s}{a}.
	 \end{equation*}
	Since it is sufficient to prove (\ref{estimp}) for $p$ large enough,  we pick any $p'$ such that
	 \begin{equation}\label{e20}
	 	1 < p' < \min\abra{\frac{n+2}{2}, \frac{2}{n+2} - \frac{b-s}{a}}.
	 \end{equation}
\medskip

	 Let $\theta \in L^{p'}(Q_T)$ with $\theta \geq 0$ and $\|\theta\|_{L^{p'}(Q_T)} \leq 1$. Let $\phi$ be the unique nonnegative solution to
the equation
	 \begin{equation}\label{e21}
	 \left\{
	 \begin{aligned}
	 	\pa_t \phi + \Delta \phi &= -\theta, &&(x,t)\in Q_T,\\
	 	\na\phi \cdot \nu &= 0, &&(x,t)\in \pa\Omega\times(0,T),\\
	 	\phi(x,T) &= 0, &&x\in\Omega.
	 \end{aligned}
	 \right.
	 \end{equation}
	 To avoid any possible confusion, we remark that \eqref{e21} is a forward heat equation for the function $\psi(x,s)=\phi(x,T-s)$, with respect to the new time variable $s=T-t$. From maximal regularity in \cite{ladyvzenskaja1988linear} and Lemma \ref{embedding},
we know that there exists $C_{p'}>0$ (depending on $\Omega$, $T$, $n$ and $p$ only) such that
	 \begin{equation}\label{e22}
	 	\|\pa_t\phi\|_{L^{p'}(Q_T)} + \|\phi\|_{L^{q_2}(Q_T)} + \|\na \phi\|_{L^{q_3}(Q_T)} \leq C_{p'} , 
	 \end{equation}
	 where
	 \begin{equation}\label{q2q3}
		 q_2 = \frac{(n+2)p'}{n+2-2p'} \quad \text{ and } \quad q_3 = \frac{(n+2)p'}{n+2-p'}.
	 \end{equation}
	 Set $p = \frac{p'}{p'-1}$. Then, by using the equation satisfied by $m$ in \eqref{e11_r} and integrating by parts, we have
	 \begin{equation}\label{e23}
	 \begin{aligned}
		 \int_0^T\int_{\Omega}m\theta dxdt&= \int_0^T\int_{\Omega} m (-\pa_t\phi - \Delta \phi)dxdt\\
		 &= \lambda\int_{\Omega}m_0\phi(0)dx + \int_0^T\int_{\Omega}\phi(\pa_t m - \Delta m)dxdt\\
		 &\leq \int_{\Omega}m_0\phi(0)dx + \lambda\int_0^T\int_{\Omega} \phi m dxdt + \chi\int_0^T\int_{\Omega}f_r(m)\na \phi \na c dxdt\\
		 &\leq \int_{\Omega}m_0\phi(0)dx + (I) + (II) ,
	 \end{aligned}
	 \end{equation}
	 where
	 \begin{equation*}
	 	(I) = \int_0^T\int_{\Omega}\phi m dxdt \quad \text{ and } \quad (II) = \chi\int_0^T\int_{\Omega}|f_r(m)||\na \phi||\na c|dxdt.
	 \end{equation*}
	 For the first term on the right-hand side of \eqref{e23}, we use
	 \begin{equation*}
		 \|\phi(0)\|_{L^{p'}(\Omega)}^{p'} = \int_{\Omega}\left|\int_0^T\pa_t\phi dt\right|^{p'}dx \leq T^{\frac{1}{p-1}}\|\pa_t\phi\|_{L^{p'}(Q_T)}^{p'} \leq T^{\frac{1}{p-1}}C_{p'}^{p'}
	 \end{equation*}
	 thanks to \eqref{e22}, in order to estimate
	 \begin{equation}\label{add}
		 \int_{\Omega}m_0\phi(0)dx \leq \|m_0\|_{L^p(\Omega)}\|\phi(0)\|_{L^{p'}(\Omega)} \leq \|m_0\|_{L^p(\Omega)}T^{\frac 1p}C_{p'}.
	 \end{equation}
 For the treatment of  $(I)$, we consider two cases.
	 \begin{itemize}
	 \item {\bf Case 1.} $a\geq \frac{n+2}{2}$. Note that 
	 \begin{equation*}
	 	a \geq \frac{n+2}{2} \Longrightarrow a > \frac{(n+2)p'}{(n+2)(p'-1) + 2p'} \Longrightarrow \frac{a}{a-1} < \frac{(n+2)p'}{n+2-2p'} = q_2.
	 \end{equation*}
	 Therefore, by using H\"older's inequality, we have (for some $C_T$ depending only on $\Omega$, $T$, $a$, $n$ and $p$)
	 \begin{equation}\label{e24}
		 (I) \leq \|\phi\|_{L^{\frac{a}{a-1}}(Q_T)}\|m\|_{L^a(Q_T)} \leq C_T\, \|\phi\|_{L^{q_2}(Q_T)}\|m\|_{L^{a}(Q_T)} \leq C_T\,C_{p'}\, [K_1\,(T+1)]^{1/a},
	 \end{equation}
	 thanks to \eqref{e15} and \eqref{e22}. 
	 
	 \item {\bf Case 2.} $a< \frac{n+2}{2}$. Define $\vartheta = 1 - \frac{2a}{n+2}$ and note that $0<\vartheta<1$. From H\"older's inequality, there exists $R_{T,|\Omega|}>0$ (depending only on $\Omega$, $T$, $a$, $n$ and $p$) such that
	 \begin{equation}\label{e25}
	 \begin{aligned}
		 (I) &=  \int_0^T\int_{\Omega} \phi m^{1-\vartheta}m^{\vartheta}dxdt\\
		 &\leq R_{T,|\Omega|}\bra{\int_0^T\int_{\Omega}\phi^{p'}m^{p'(1-\vartheta)}}^{\frac{1}{p'}}\|m\|_{L^p(Q_T)}^{\vartheta}\\
		 &\leq R_{T,|\Omega|} \, \|\phi\|_{L^{\frac{(n+2)p'}{n+2-2p'}}(Q_T)}\bra{\int_0^T\int_{\Omega}m^{\frac{(n+2)(1-\vartheta)}{2}}dxdt}^{\frac{2}{n+2}}\|m\|_{L^p(Q_T)}^{\vartheta}\\
		 &= R_{T,|\Omega|} \, \|\phi\|_{L^{q_2}(Q_T)}\|m\|_{L^a(Q_T)}^{(1-\vartheta)}\|m\|_{L^p(Q_T)}^{\vartheta},
	 \end{aligned}
	 \end{equation}
	 using 
	 \begin{equation}\label{vartheta}
		 q_2 = \frac{(n+2)p'}{n+2-2p'} \quad \text{ and } \quad \frac{(n+2)(1-\vartheta)}{2} = a.
	 \end{equation}
	 Therefore,
	 \begin{equation}\label{e26}
		 (I) \leq R_{T,\Omega} \, C_{p'}\sbra{K_1(T+1)}^{\frac{1-\vartheta}{a}}\|m\|_{L^p(Q_T)}^{\vartheta} .
	 \end{equation}
	\end{itemize}
	 From these two cases, or more precisely, from \eqref{e24} and \eqref{e26}, we have
	 \begin{equation}\label{estimate_I}
	 	(I) \leq C_T^*\,\bra{1+\|m\|_{L^p(Q_T)}^{\vartheta}}
	 \end{equation}
	 for some $0<\vartheta<1$ (depending on $a,n$), and $C_T^*$ (depending on $\Omega$, $T$, $n$, $p$, $a$, $\beta$, $\delta$, $\|m_0\|_{L^1(\Omega)}$, $\|c_0\|_{L^1(\Omega)}$, $\|d_0\|_{L^{\infty}(\Omega)}$,  but not depending on $\lambda$).
\medskip

 In order to estimate $(II)$, we write
	 \begin{equation}\label{e27}
	 \begin{aligned}
		 (II) &= \chi\int_0^T\int_{\Omega}|f_r(m)||\na \phi||\na c|dxdt\\
		 &\le \chi \gamma \int_0^T\int_{\Omega}|m|^b|\na \phi||\na c|dxdt\\
		 &\leq \chi \gamma \bra{\int_0^T\int_{\Omega}\bra{m^{b-s}|\na \phi| |\na c|}^{p'}dxdt}^{\frac{1}{p'}}\|m\|_{L^p(Q_T)}^{s}\\
		 &= \chi \gamma \bra{\int_0^T\int_{\Omega}m^{p'(b-s)}|\na \phi|^{p'} |\na c|^{p'}dxdt}^{\frac{1}{p'}}\|m\|_{L^p(Q_T)}^{s}.
	 \end{aligned}
	 \end{equation}
	 Note that $1<p'<n+2$ (see \eqref{e20}). We can use H\"older's inequality to estimate further
	 \begin{equation*}
	 \begin{aligned}
	 	(II) &\leq \chi\gamma \, \|\na \phi\|_{L^{\frac{(n+2)p'}{n+2-p'}}(Q_T)}\bra{\int_0^T\int_{\Omega}m^{(n+2)(b-s)}|\na c|^{n+2}dxdt}^{\frac{1}{n+2}}\|m\|_{L^p(Q_T)}^s\\
	 	&\le \chi\gamma \,  \|\na \phi\|_{L^{q_3}(Q_T)}\|m\|_{L^a(Q_T)}^{b-s}\|\na c\|_{L^{q_4}(Q_T)}\|m\|_{L^p(Q_T)}^s,
	 \end{aligned}	
	 \end{equation*}
	 where
	 \begin{equation}\label{q4}
		 q_4 = \frac{(n+2)a}{a-(n+2)(b-s)}
	 \end{equation}
	 is well defined since $(n+2)(b-s) < a$ (see \eqref{e3}). Therefore, we can apply \eqref{e22}, \eqref{e15} in order to get 
	 \begin{equation}\label{e28}
		 (II) \leq \chi \gamma \, C_{p'}\, \sbra{K_1(T+1)}^{\frac{b-s}{a}}\|\na c\|_{L^{q_4}(Q_T)}\|m\|_{L^p(Q_T)}^s.
	 \end{equation}
	 Note that from \eqref{q4} we see that when $a<n+2$,
	 \begin{equation}\label{q4<q1}
		 q_4 = \frac{(n+2)a}{a-(n+2)(b-s)}<\frac{(n+2)a}{n+2-a} = q_1,
	 \end{equation}
	 and that when $a\geq n+2$, $q_4<q_1$ since $q_1<+\infty$ can be chosen arbitrarily. Therefore (for some $C_T>0$ depending only on $\Omega$, $T$, $a$, $b$, $n$),
	 \begin{equation*}
		 \|\na c\|_{L^{q_4}(Q_T)} \leq C_T\|\na c\|_{L^{q_1}(Q_T)} \leq C_T\, K_2\, K_1\,(T+1),
	 \end{equation*}
	 thanks to \eqref{e16}. It then follows from \eqref{e28} that
	 \begin{equation}\label{estimate_II}
		 (II) \leq D_T^*\, \|m\|_{L^p(Q_T)}^s,
	 \end{equation}
	 where $D_T^*>0$ only depends on $\Omega$, $T$, $n$, $p$, $a$, $b$, $\beta$, $\delta$, $\var_0$, $\chi$, $\gamma$, $\|m_0\|_{L^1(\Omega)}$, $\|c_0\|_{W^{\tilde p, (2-2/\tilde{p})}(\Omega)}$, $\|d_0\|_{L^{\infty}(\Omega)}$,  but does not depend on $\lambda$.
	 \medskip

	 From \eqref{e23}, \eqref{add}, \eqref{estimate_I} and \eqref{estimate_II}, we get (for $E^* := C_{p'} T^{1/p} + C_T^* + D_T^*$) 
	 \begin{equation*}
		 \int_0^T\int_{\Omega}m\theta dxdt \leq E^*\, \sbra{\|m_0\|_{L^p(\Omega)} + 1 + \|m\|_{L^p(Q_T)}^{\vartheta} + \|m\|_{L^p(Q_T)}^s}
	 \end{equation*}
	 for all $0\leq \theta \in L^{p'}(Q_T)$ satisfying $\|\theta\|_{L^{p'}(Q_T)} = 1$. 
Therefore, by duality, we get
	 \begin{equation*}
		 \|m\|_{L^{p}(Q_T)} \leq E^*\,\sbra{\|m_0\|_{L^p(\Omega)} + 1 + \|m\|_{L^p(Q_T)}^{\vartheta} + \|m\|_{L^p(Q_T)}^s}.
	 \end{equation*}
	 Note that $0 < \vartheta, s < 1$. Thus we can use Young's inequality in order to get
	 \begin{equation}\label{e32}
		 \|m\|_{L^{p}(Q_T)}\leq F_p^*\, \bra{1+\|m_0\|_{L^p(\Omega)}},
	 \end{equation}
where $F_p^*$ depends on the same quantities as $D_T^*$.
\par 
	 From \eqref{e20}, we recall from $p=\frac{p'}{p'-1}$ that $1<p<\infty$ can be chosen arbitrarily. We also insist that the bound in \eqref{e32} does not depend on $\lambda \in (0,1]$. 
	 \end{proof}
 	
 We now use Lemma \ref{lem:m_Lp} in order to conclude the proof of Prop. \ref{prop11}.
\medskip

	 Since $\tilde p>n+2$, Lemma \ref{embedding}) guarantees $W^{2,1}_{\tilde p}(Q_T)\hookrightarrow \Y$ compactly.  Applying this to the equation satisfied by $c$ in \eqref{e11_r}, we get (for some constant $C_T>0$ depending only on $\Omega$, $T$, $n$, $p$  and $\var_0$) 
	 \begin{equation}\label{add1}
		 \|c\|_{L^{\infty}(Q_T)} 
 \leq C_T\, \bigg( \|\lambda(\delta d + \beta m)\|_{L^{\tilde p}(Q_T)}  + \|c_0\|_{W^{\tilde p,(2-2/\tilde{p})}(\Omega)} \bigg) \leq \S_T,
	 \end{equation}
for some $\S_T>0$ (depending only on the same quantities as $D_T^*$, except $\tilde p$).
\medskip

 Applying the heat semigroup property and maximal regularity, we can find $\S_T^*$
 (depending on the same quantities as $D_T^*$) such that
 \begin{equation}\label{add1bis}
\| \na c\|_{L^{\infty}(Q_T)}   + \|c\|_{W^{2,1}_{\tilde p}(Q_T)} \le \S_T^*,
\end{equation}
and therefore
\begin{equation}\label{add0}
	\|c\|_{\Y} 
\leq S_T + \S_T^*.
\end{equation}

	We now consider the equation satisfied by $m$ in \eqref{e11_r}: 
	 \begin{equation}\label{add2}
	 \begin{aligned}
		 \pa_t m
		 &= \Delta m + \lambda m(1-m^{a-1}) - \chi \na\cdot(f_r(m)\na c)\\
		 &= \Delta m -\chi f_r'(m)\na c\cdot \na m + \lambda m(1-m^{a-1}) - \chi f_r(m)\Delta c.
	 \end{aligned}
	 \end{equation}
	 From \eqref{e32} with $p=\tilde p$, \eqref{add0}, and assumption \eqref{assumption_f}, we apply the properties of regularity for parabolic equations (cf \cite{ladyvzenskaja1988linear}) to \eqref{add2}, and get
	 \begin{equation*}
	 	\|m\|_{W^{2,1}_{\tilde p}(Q_T)} \leq C_T\left(\left\|\lambda m(1-m^{a-1}) - \chi f_r(m)\Delta c\right\|_{L^{\tilde p}(Q_T)}+\|m_0\|_{W^{\tilde p,(2-2/\tilde{p})}(\Omega)}\right) \leq \S_T^{**},
	 \end{equation*}
 for some $\S_T^{**}>0$ depending on the same parameters as $\S_T$. By the embedding \eqref{em}
	 \begin{equation}\label{add3}
	 	\|m\|_{\Y} \leq C_T\|m\|_{W^{2,1}_{\tilde p}(Q_T)} \leq C_T\, \S_T^{**}.
	 \end{equation}
	 Note that the bounds \eqref{add0} and \eqref{add3} do not depend on $\lambda \in (0,1]$ (or $(m,c)$). This means that the set $\Z$ defined in \eqref{def_Z} is bounded in $\X=\Y\times\Y$. This shows
that the last assumption of Leray-Schauder fixed point theorem holds, and  therefore that the mapping $\F_r$ has a fixed point, which satisfies as a
consequence system (\ref{e11_r}), together with estimates (\ref{add0}), (\ref{add3}). Since moreover the upper bound $\S_T$ in 
 \eqref{add3} does not depend on  $r$, by taking $r\geq \S_T$, we obtain a solution to the original system \eqref{sys_original} (for all $T>0$). This solution is strong in the sense that all terms appearing in the system are defined a.e., (as $L^1(Q_T)$ functions).

\subsection{Uniqueness} \label{uniqueness}	
\ 

\medskip
	 
In this subsection, we write down a stability result for strong solutions of system (\ref{sys_original}), which entails the uniqueness result in Theorem \ref{thm:main}. More precisely, we write the

\begin{proposition} \label{prop12}
Let $T>0$ and let $\Omega$ be a smooth ($C^{2+ \alpha}$ for some $\alpha>0$) bounded connected open set of $\R^n$ (for $n \in \N - \{0\}$). In addition, assume $a>1$, $\chi, \eps_0, \delta, \beta > 0$, $\gamma, b, l \ge 0$, and $f$ and $g$ satisfying (F) and (G). We consider two sets of nonnegative (for each component), initial data $(m_{01}, c_{01}, d_{01})$ and $(m_{02}, c_{02}, d_{02})$ in $W^{(2-2/{\tilde p}),\tilde p}(\Omega)\times W^{(2-2/{\tilde p}),\tilde p}(\Omega) \times L^\infty(\Omega)$ for some $\tilde p>\max(n+2,a)$, and two sets of nonnegative (for each component) strong (in the sense of Theorem \ref{thm:main}) solutions $(m_{1}, c_{1}, d_{1})$ and $(m_{2}, c_{2}, d_{2})$ to  system \eqref{sys_original} (with corresponding initial data) on $Q_T$, satisfying estimate (\ref{pourun}). We denote $\mu := \|d_2\|_{L^{\infty}(Q_T)}$, $\mu_c :=  \|\na c_1\|_{L^{\infty}(Q_T)}$, $\mu_m := \max( \|m_1\|_{L^{\infty}(Q_T)}, \|m_2\|_{L^{\infty}(Q_T)})$, $\mu_f := \|f\|_{L^{\infty}([0, \mu_m])}$,
$\mu_{f'} := \|f'\|_{L^{\infty}([0, \mu_m])}$, $\mu_{g'} := \|g'\|_{L^{\infty}([0, \mu_m])}$.
\par 
Then, for all $t \in [0,T]$, 
\begin{equation}\label{ee}
 \int_{\Omega} |m_1(\cdot, t) - m_2(\cdot, t)|^2dx + \frac{\chi^2\, \mu_f^2}{2\var_0} \int_{\Omega} |c_1(\cdot, t) - c_2(\cdot, t)|^2dx + \int_{\Omega} |d_1(\cdot, t) - d_2(\cdot, t)|^2dx 
\end{equation}
$$ \le  e^{G\,t} \,\, \bigg(  \int_{\Omega} |m_{10} - m_{20}|^2dx + \frac{\chi^2\, \mu_f^2}{2\var_0} \int_{\Omega} |c_{10} - c_{20}|^2dx + \int_{\Omega} |d_{10} - d_{20}|^2dx \bigg), $$
where
$$ G := \max\bigg( 2 + \chi^2\,\mu_{f'}^2 \, \mu_c^2 + \frac{\chi^2\,\beta\,\mu_f^2}{2\var_0} + \mu_{g'}\,(1+\mu) ;  \beta + \delta ; \frac{\chi^2\,\delta\,\mu_f^2}{2\var_0} + (1+\mu)  \bigg) . $$
\end{proposition}

\begin{proof}
Substracting the equations satisfied by $m_1$ and $m_2$, $c_1$ and $c_2$, and finally $d_1$ and $d_2$, and then performing integrations by parts, we end up with the identities:
\begin{equation}\label{eem}
	\begin{aligned}
&\frac12 \frac{d}{dt} \int_{\Omega} |m_1 - m_2|^2dx + \int_{\Omega} |\nabla(m_1 - m_2)|^2dx + \int_{\Omega} (m_1^a - m_2^a)\,(m_1 - m_2)dx\\
&=  \int_{\Omega} |m_1 - m_2|^2dx + \chi \int_{\Omega} f(m_2)\, \nabla(m_1 - m_2)\cdot \nabla(c_1 - c_2)dx
\\&\quad + \chi \int_{\Omega}(f(m_1) -  f(m_2))\, \nabla(m_1 - m_2)\cdot \nabla c_1dx ,
\end{aligned}
\end{equation}
  \begin{equation}\label{eec}
  	\begin{aligned}
&\frac12 \frac{d}{dt} \int_{\Omega} |c_1 - c_2|^2dx + \var_0 \int_{\Omega} |\nabla(c_1 - c_2)|^2dx + \int_{\Omega} |c_1 - c_2|^2dx\\
&=\delta \int_{\Omega} (c_1 - c_2)\,(d_1 -  d_2)dx 
+ \beta \int_{\Omega} (c_1 - c_2)\,(m_1 -  m_2)dx  ,
\end{aligned}
\end{equation}
 \begin{equation}\label{eed}
 	\begin{aligned}
&\frac12 \frac{d}{dt} \int_{\Omega} |d_1 - d_2|^2dx   + \int_{\Omega} g(m_1)\, |d_1 - d_2|^2dx\\
& = -  \int_{\Omega} (1 - d_2)\,(d_1 -  d_2)\,(g(m_1) - g(m_2))dx.
\end{aligned}
\end{equation} 

Then we perform the following estimates (the last one uses Young's inequality):
 \begin{equation}\label{eede}
\frac{d}{dt} \int_{\Omega} |d_1 - d_2|^2dx   \le (1 +\mu)\,  \int_{\Omega} |d_1 -  d_2|^2dx + (1 +\mu)\, \mu_{g'}^2  \int_{\Omega} |m_1 - m_2|^2 dx ,
\end{equation} 
 \begin{equation}\label{eece}
 \begin{aligned}
\frac{d}{dt} \int_{\Omega} |c_1 - c_2|^2dx + 2\var_0 \int_{\Omega} |\nabla(c_1 - c_2)|^2dx\\
\le \delta\,  \int_{\Omega} |d_1 -  d_2|^2dx + \beta \int_{\Omega} |m_1 - m_2|^2dx + (\beta + \delta) \int_{\Omega} |c_1 - c_2|^2dx,  
\end{aligned}
\end{equation}
\begin{equation}\label{eeme}
 \frac{d}{dt} \int_{\Omega} |m_1 - m_2|^2dx  \le (2 + \chi^2\,\mu_{f'}^2 \mu_c^2)  \int_{\Omega} |m_1 - m_2|^2dx + \chi^2  \mu_f^2 \int_{\Omega} |\nabla(c_1 - c_2)|^2dx . 
\end{equation}
Estimate (\ref{ee}) is directly obtained from these estimates by an application of Gronwall's lemma to the quantity
$$  \int_{\Omega} |m_1 - m_2|^2dx  +  \frac{\chi^2\, \mu_f^2}{2\var_0}  \int_{\Omega} |c_1 - c_2|^2dx + \int_{\Omega} |d_1 - d_2|^2dx .$$
\medskip

Note that uniqueness in Theorem \ref{thm:main} is a direct consequence of Proposition \ref{prop12}. Moreover, since $T$ is arbitrary in Propositions \ref{prop11} and \ref{prop12}, one can build the unique solution of system (\ref{sys_original}) on $\R_+$ that appears in the conclusion of Theorem \ref{thm:main} by patching together the solutions defined on finite time intervals. 
\end{proof}

 \subsection{Uniform-in-time bounds}\label{sec:uniform}
\

\medskip

 In this section, we will conclude the proof of Theorem \ref{thm:main} by showing that the unique solution to \eqref{sys_original} obtained  in Propositions \ref{prop11} and \ref{prop12} is globally (w.r.t. time) bounded.  Thanks to Proposition \ref{prop11}, it is sufficient to show that  the sequences $\left\{\|m\|_{L^\infty(Q_{\tau,\tau+1})}\right\}_{\tau=1}^\infty$ and $\left\{\|c\|_{L^\infty(Q_{\tau,\tau+1})}\right\}_{\tau=1}^\infty$ are bounded.
As a consequence, we will obtain estimates in this subsection in which all constants {\em do not depend on  $\tau$}. In addition, we will see that the $\limsup$ of each sequence is independent of $m_0$ and $c_0$.
 \medskip 
 
 We first recall the uniform in time bound
 \begin{equation*}
	\sup_{t \in \R_+} \|d(\cdot, t)\|_{L^\infty(\Omega)} \leq \mu :=  \max\abra{1, \|d_0\|_{L^\infty(\Omega)}}.
 \end{equation*}
 Coming back to the proof of Lemma \ref{l21}, and specializing it in the case when $\lambda =1$, we see that
$$ \frac{d}{dt} \int_{\Omega} (m+c) \le |\Omega|\,(\delta\,\mu + k_a) - \int_{\Omega} (m+c), $$
so that
 \begin{equation}\label{La-local}
	 \sup_{t\in \R_+}\bra{\|m(t)\|_{L^1(\Omega)} + \|c(t)\|_{L^1(\Omega)}} \leq K_1^*,
 \end{equation}
where $K_1^* := \max(  |\Omega|\,(\delta\,\mu + k_a) , \|m_0\|_{L^1(\Omega)} + \|c_0\|_{L^1(\Omega)})$, and
 \begin{equation}\label{La-largetime}
	 \limsup_{t\to \infty}\bra{\|m(t)\|_{L^1(\Omega)} + \|c(t)\|_{L^1(\Omega)}} \leq  |\Omega|\,(\delta\,\mu + k_a).
 \end{equation}
We also recall that 
$k_a$ and $\mu$ only depend on $\beta, a$ and $\|d_0\|_{L^{\infty}(\Omega)}$. 
\medskip

Then, we prove the

 \begin{lemma}\label{uniform_m_La}
Under the assumptions of Theorem \ref{thm:main}, we consider
  $(m,c,d)$, the unique solution to \eqref{sys_original} on $\Omega \times \R_+$. Then
 	\begin{equation}\label{m_La}
	 	\sup_{\tau \in \mathbb N}\|m\|_{L^a(Q_{\tau,\tau+2})} \leq (3\,K_1^*)^{\frac 1a},
 	\end{equation}
and
\begin{equation}\label{eventual-m_La}
	 	\limsup_{\tau \to \infty}\|m\|_{L^a(Q_{\tau,\tau+2})} \leq (3\,|\Omega|\,(\delta\,\mu + k_a))^{\frac 1a}.
 	\end{equation}
 \end{lemma}

 \begin{proof}
 	By integrating the equation satisfied by $m$ on $Q_{\tau,\tau+2}$, we have
 	\begin{equation*}
	 	\int_{\Omega}m(x,\tau+2)dx = \int_{\Omega}m(x,\tau)dx + \int_{\tau}^{\tau+2}\int_{\Omega}m(x,s)(1-m(x,s)^{a-1})dxds.
 	\end{equation*}
 	Thus
 	\begin{equation*}
	 	\int_{\tau}^{\tau+2}\int_{\Omega}m(x,s)^{a}\, dxds \leq \|m(\tau)\|_{L^1(\Omega)} + \int_{\tau}^{\tau+2}\int_{\Omega}m(x,s)dx ds,
 	\end{equation*}
 	which proves both \eqref{m_La} and \eqref{eventual-m_La}.
 \end{proof}

	\begin{lemma}\label{uniform_gradient_c_Lq1}
		Under the assumptions of Theorem \ref{thm:main}, we consider
  $(m,c,d)$, the unique solution to \eqref{sys_original} on $\Omega \times \R_+$.
Let $q_1$ be defined as in \eqref{q1}. Then, there exists $K_3>0$ (depending on $\Omega$,  $n$, $a$, $\var_0, \beta, \delta$ and $\|m_0\|_{L^{1}(\Omega)}$, $\|c_0\|_{L^{1}(\Omega)}$, $\|d_0\|_{L^{\infty}(\Omega)}$, but not on $\tau$) and $\tilde K_3>0$ (depending on $\Omega$,  $n$, $a$, $\var_0, \beta, \delta$ and $\|d_0\|_{L^{\infty}(\Omega)}$, but not on $\tau$) such that
		\begin{equation}\label{gradient_c_Lq1}
			\sup_{\tau\in\mathbb N - \{0\} }\|\na c\|_{L^{q_1}(Q_{\tau,\tau+1})} \leq K_3,
		\end{equation}
         and
         \begin{equation}\label{eventual-gradient_c_Lq1}
			\limsup_{\tau\to\infty }\|\na c\|_{L^{q_1}(Q_{\tau,\tau+1})} \leq \tilde K_3.
		\end{equation}
	\end{lemma}

	\begin{proof}
		We define a smooth cutoff function $\psi\in C^\infty(\R; [0,1])$ such that $\psi(s) = 0$ for $s\leq 0$, and $\psi(s) =1$ for $s\geq 1$. Moreover, we assume that $|\psi'(s)| \leq 2$ for all $s\in \R$. For any $\tau\in \mathbb N$, the shifted cutoff function is defined by $\psi_\tau(\cdot) := \psi(\cdot-\tau)$. By multiplying the equation satisfied by $c$ by $\psi_\tau$, we get
		\begin{equation}\label{eq_c_cutoff}
			\pa_t(\psi_\tau c) = \eps_0 \Delta(\psi_\tau c) + \psi_\tau' c - \psi_\tau c + \psi_\tau[\delta d + \beta m],
		\end{equation}
		and $(\psi_\tau c)(x,\tau) = 0$. 	
		By using the  semigroup properties of the heat equation (Lemma \ref{embedding}), we see that (for some $C>0$ depending only on $\Omega$, $n$, $a$, 
$\var_0$, but not $\tau$)
		\begin{equation}\label{f1}
			\begin{aligned}
			\|\psi_\tau c\|_{L^{q_5}(Q_{\tau,\tau+2})} \leq C\bra{\|\psi_\tau'c\|_{L^a(Q_{\tau,\tau+2})} + \|\psi_\tau c\|_{L^a(Q_{\tau,\tau+2})} + \|\psi_\tau(\delta d + \beta m)\|_{L^a(Q_{\tau,\tau+2})}} ,
			\end{aligned}
		\end{equation}
		where
		\begin{equation*}
			q_5 = \left\{
			\begin{aligned}
				&\frac{(n+2)a}{n+2-2a} &&\text{ if } a<\frac{n+2}{2},\\
				&<+\infty \text{ arbitrary} &&\text{ if } a = \frac{n+2}{2},\\
                   & = +\infty  &&\text{ if } a > \frac{n+2}{2}  .
			\end{aligned}
			\right.
		\end{equation*}
		In any case, we have $a < q_5$. By using the boundedness of $\psi_\tau'$ and $\psi_\tau$, and H\"older's inequality, we have, for 
 $\alpha \in (0,1)$ satisfying $\frac{1}{a} = \frac{1-\alpha}{1} + \frac{\alpha}{q_5}$,
		\begin{equation}\label{f2}
			\|\psi_\tau' c\|_{L^a(Q_{\tau,\tau+2})} \leq 2\|c\|_{L^a(Q_{\tau,\tau+2})} \leq 2\|c\|_{L^1(Q_{\tau,\tau+2})}^{1-\alpha}\|c\|_{L^{q_5}(Q_{\tau,\tau+2})}^{\alpha} \leq 2(2K_1^*)^{1-\alpha}\,\|c\|_{L^{q_5}(Q_{\tau,\tau+2})}^{\alpha},
		\end{equation}
		and similarly
		\begin{equation}\label{f3}
			\|\psi_\tau c\|_{L^a(Q_{\tau,\tau+2})} \leq (2K_1^*)^{1-\alpha}\,\|c\|_{L^{q_5}(Q_{\tau,\tau+2})}^{\alpha} .
		\end{equation}
  From the uniform boundedness of $d$ and \eqref{m_La}, we see that
		\begin{equation}\label{f4}
			\|\psi_\tau(\delta d+\beta m)\|_{L^a(Q_{\tau,\tau+2})} \leq (3 K_1^*)^{\frac 1a}\beta + \delta\,\mu\,(2|\Omega|)^{1/a}.
		\end{equation}
		Inserting \eqref{f2}--\eqref{f4} into \eqref{f1} and using $\psi_{\tau}\equiv 1$ on $(\tau+1,\tau+2)$ yields, for some constant 
 $C_2>0$ (depending only on $\Omega$, $n$, $\delta$, $\beta$, $a$, $\var_0$, $\|m_0\|_{L^{1}(\Omega)}$, $\|c_0\|_{L^{1}(\Omega)}$, 
$\|d_0\|_{L^{\infty}(\Omega)}$, but not $\tau$)
		\begin{equation}\label{f5}
			\|c\|_{L^{q_5}(Q_{\tau+1,\tau+2})} \leq C_2\,\|c\|_{L^{q_5}(Q_{\tau,\tau+2})}^{\alpha} + C_2.
		\end{equation}
		Denote by $\Lambda:= \abra{\tau+1 \in \mathbb N\;:\;\|c\|_{L^{q_5}(Q_{\tau,\tau +1})} \leq \|c\|_{L^{q_5}(Q_{\tau +1,\tau+2})}}$. Then for all $\tau \in \Lambda$, we get from \eqref{f5} that
		\begin{equation*}
			\|c\|_{L^{q_5}(Q_{\tau+1,\tau+2})} \leq 2C_2\,\|c\|_{L^{q_5}(Q_{\tau+1,\tau+2})}^{\alpha} + C_2 \leq \frac 12\|c\|_{L^{q_5}(Q_{\tau+1,\tau+2})} + C_3,
		\end{equation*}
 where $C_3>0$ only depends on the same parameters as $C_2$.
		Thus
		\begin{equation*}
			\|c\|_{L^{q_5}(Q_{\tau+1,\tau+2})} \leq 2C_3 \quad \text{ for all } \quad \tau \in \Lambda.
		\end{equation*}
		By the definition of $\Lambda$, this in fact implies that
		\begin{equation*}
		\|c\|_{L^{q_5}(Q_{\tau+1,\tau+2})} \leq 2C_3 \quad \text{ for all } \quad \tau \in \mathbb N.
		\end{equation*}
		Since $a<q_5$, we get with $\frac 1a = \frac{1-\alpha}{1} + \frac{\alpha}{q_5}$,
		\begin{equation}\label{La_c}
			\sup_{\tau\in\mathbb N - \{0\} }\|c\|_{L^a(Q_{\tau,\tau+1})} \leq \sup_{\tau\in\mathbb N - \{0\} }\|c\|_{L^1(Q_{\tau,\tau+1})}^{1-\alpha}\|c\|_{L^{q_5}(Q_{\tau,\tau+1})}^{\alpha}\leq  (K_1^*)^{1-\alpha}(2C_3)^{\alpha} .
		\end{equation}
		We can now use this estimate, the uniform boundedness of $d$, and \eqref{m_La}, and apply the semigroup 
properties of the heat kernel to \eqref{eq_c_cutoff}, and get (for some constant $C>0$ only depending on $\Omega$, $n$, $\varepsilon_0$,
and $a$),
		\begin{equation*}
			\|\psi_\tau|\na c|\|_{L^{q_1}(Q_{\tau,\tau+2})} \leq C\, \|\psi_\tau' c - \psi_\tau c + \psi_\tau[\delta d + \beta m]\|_{L^a(Q_{\tau,\tau+2})}
		\end{equation*}
$$ \le C\, (6(K_1^*)^{1-\alpha} (2C_3)^{\alpha} + (3 K_1^*)^{\frac 1a}\beta + \delta\,\mu\, (2|\Omega|)^{1/a} ). $$
		Finally, by using $\psi_\tau \equiv 1$ on $(\tau+1,\tau+2)$, we obtain \eqref{gradient_c_Lq1}. Similarly, using \eqref{eventual-m_La}, we obtain \eqref{eventual-gradient_c_Lq1}.
	\end{proof}

	Using Lemmas \ref{uniform_m_La} and \ref{uniform_gradient_c_Lq1}, we obtain the following result:

	\begin{lemma}\label{uniform_m_Lp}
   Under the assumptions of Theorem \ref{thm:main}, we consider
  $(m,c,d)$, the unique solution to \eqref{sys_original} on $\Omega \times \R_+$.
		For any $1 \le p<\infty$, there exist a constant $K_p>0$ (depending on $\Omega$, $n$, $p$,  $a, \beta, \delta$, $\var_0$ and $\|m_0\|_{L^1(\Omega)}$, $\|c_0\|_{L^1(\Omega)}$, $\|d_0\|_{L^{\infty}(\Omega)}$, but not on $\tau$), and a constant $\tilde K_p>0$ (depending on $\Omega$, $n$, $p$,  $a, \beta, \delta$, $\var_0$ and  $\|d_0\|_{L^{\infty}(\Omega)}$, but not on $\tau$) such that
		\begin{equation}\label{es_uniform_m_Lp}
			\sup_{\tau\in\mathbb N - \{0\} }\|m\|_{L^p(Q_{\tau,\tau+1})} \leq K_p,
		\end{equation}
         and
         \begin{equation}\label{eventual-es_uniform_m_Lp}
			\limsup_{\tau\to\infty }\|m\|_{L^p(Q_{\tau,\tau+1})} \leq \tilde K_p.
		\end{equation}
	\end{lemma}

	\begin{proof}
		The proof of this lemma is similar to that of Lemma \ref{lem:m_Lp}, except that we will consider the cylinder $Q_{\tau,\tau+2}$ instead of $Q_T$. More precisely, by multiplying the equation satisfied by 
$m$ in \eqref{sys_original} with $\psi_\tau$, we get
		\begin{equation}\label{eq_m_cutoff}
			\pa_t(\psi_\tau m) = \Delta(\psi_\tau m) + \psi_\tau' m + \psi_\tau m(1-m^{a-1}) - \chi \psi_\tau \na\cdot(f(m)\na c),
		\end{equation}
		and $(\psi_\tau m)(x,\tau) = 0$. Let $s\in (0,1)$ satisfy \eqref{e3}, and choose $p'$ as in \eqref{e20}. Consider $0\leq \theta \in L^{p'}(Q_{\tau,\tau+2})$ with $\|\theta\|_{L^{p'}(Q_{\tau,\tau+2})} = 1$. Let $\phi$ be the unique nonnegative solution to the equation
		\begin{equation*}
			\begin{cases}
				\pa_t\phi + \Delta \phi = -\theta, &(x,t)\in Q_{\tau,\tau+2},\\
				\na\phi\cdot \nu = 0, &(x,t)\in \pa\Omega\times(\tau,\tau+2),\\
				\phi(x,\tau+2) = 0, &x\in\Omega.
			\end{cases}
		\end{equation*}
		From  the semigroup properties of the heat equation (Lemma \ref{embedding}),
 we know that there exists a constant $L_{p'}>0$ (depending on $\Omega$, $n$, $p$ but not $\tau$) such that
		\begin{equation}\label{f6}
			\|\phi\|_{L^{q_2}(Q_{\tau,\tau+2})} + \|\na \phi\|_{L^{q_3}(Q_{\tau,\tau+2})} \leq L_{p'} ,
		\end{equation}
		with $q_2$ and $q_3$ defined in \eqref{q2q3}. By integration by parts, we estimate similarly to \eqref{e23} and get
		\begin{equation}\label{f7}
		\begin{aligned}
			\int_{\tau}^{\tau+2}\int_{\Omega}(\psi_\tau m)\theta dxdt &\leq \int_{\tau}^{\tau+2}\int_{\Omega}(\psi_\tau' m + \psi_\tau m)\,
\phi\, dxdt + \chi\int_{\tau}^{\tau+2}\int_{\Omega} |f(m)||\na \phi| |\na c|\, dxdt\\
			&=: (III) + (IV).
		\end{aligned}
		\end{equation}
		We estimate $(III)$ similarly to $(I)$ in \eqref{e23}. More precisely, in the case when $a\geq \frac{n+2}{2}$, we have (for some $C_4>0$ depending only on $\Omega$, $n$, $p$,  $a, \beta, \delta$, $\var_0$ and $\|m_0\|_{L^1(\Omega)}$, $\|c_0\|_{L^1(\Omega)}$, $\|d_0\|_{L^{\infty}(\Omega)}$)
		\begin{equation}\label{f8}
			(III) \leq 3\|\phi\|_{L^{\frac{a}{a-1}(Q_{\tau,\tau+2})}}\|m\|_{L^a(Q_{\tau,\tau+2})}
 \leq C_4,
		\end{equation}
		and in the case when  $a < \frac{n+2}{2}$, we have (for some $C_5>0$ depending only on $\Omega$, and some $C_6>0$ depending on the same quantites as $C_4$)
		\begin{equation}\label{f9}
			\begin{aligned}
			(III) \leq 3\int_{\tau}^{\tau+2}\int_{\Omega}\phi m^{1-\vartheta}m^{\vartheta}dxdt &\leq C_5\,
\|\phi\|_{L^{q_2}(Q_{\tau,\tau+2})}\|m\|_{L^a(Q_{\tau,\tau+2})}^{1-\vartheta}\|m\|_{L^p(Q_{\tau,\tau+2})}^{\vartheta}\\
			&\leq C_6\, \|m\|_{L^p(Q_{\tau,\tau+2})}^{\vartheta},
			\end{aligned}
		\end{equation}
		where $\vartheta$ is defined in \eqref{vartheta}. It follows from \eqref{f8} and \eqref{f9} that, in any case, we have the estimate
		\begin{equation}\label{estimate_III}
			(III) \leq  \max(C_4, C_6)\, \bra{1+\|m\|_{L^p(Q_{\tau,\tau+2})}^{\vartheta}}.
		\end{equation}
		The term $(IV)$ can be estimated similarly to $(II)$ in \eqref{e27} and \eqref{e28}. More precisely, it follows from H\"older's inequality that
(for some $C_7, C_8>0$ depending only on $\Omega$, and $C_9>0$  depending on the same quantites as $C_4$ and $b, \gamma, \chi$)
		\begin{equation*}
			\begin{aligned}
				(IV) &\leq \chi\gamma\int_{\tau}^{\tau+2}\int_{\Omega}|m|^b|\na \phi||\na c|dxdt\\
				&\leq \chi\gamma \bra{\int_{\tau}^{\tau+2}\int_{\Omega}m^{p'(b-s)}|\na \phi|^{p'}|\na c|^{p'}dxdt}^{\frac{1}{p'}}\|m\|_{L^p(Q_{\tau,\tau+2})}^{s}\\
				&\leq C_7\, \chi \gamma\|\na \phi\|_{L^{q_3}(Q_{\tau,\tau+2})}\|m\|_{L^a(Q_{\tau,\tau+2})}^{b-s}\|\na c\|_{L^{q_4}(Q_{\tau,\tau+2})}\|m\|_{L^p(Q_{\tau,\tau+2})}^{s} \quad \text{($q_4$ is defined in (\ref{q4}))}\\
				&\leq C_8\, \chi \gamma L_{p'}\|m\|_{L^a(Q_{\tau,\tau+2})}^{b-s}\|\na c\|_{L^{q_1}(Q_{\tau,\tau+2})}\|m\|_{L^p(Q_{\tau,\tau+2})}^{s} \quad \text{(since $q_4<q_1$, see (\ref{q4<q1}))}\\
				&\leq C_9\,\|m\|_{L^p(Q_{\tau,\tau+2})}^s.
			\end{aligned}
		\end{equation*}
		Inserting this and \eqref{estimate_III} into \eqref{f7}, we get
		\begin{equation*}
			\int_{\tau}^{\tau+2}\int_{\Omega}(\psi_\tau m)\theta dxdt \leq \max(C_4, C_6, C_9)\, \bra{1+\|m\|_{L^p(Q_{\tau,\tau+2})}^{\vartheta} + \|m\|_{L^p(Q_{\tau,\tau+2})}^s}
		\end{equation*}
		for all $0\leq \theta \in L^{p'}(Q_{\tau,\tau+2})$ satisfying $\|\theta\|_{L^{p'}(Q_{\tau,\tau+2})} = 1$. Therefore, by duality, we get
(for $C_{10} := \max(C_4, C_6, C_9))$
		\begin{equation}\label{f10}
			\|\psi_\tau m\|_{L^{p}(Q_{\tau,\tau+2})} \leq C_{10}\, \bra{1+\|m\|_{L^p(Q_{\tau,\tau+2})}^{\vartheta} + \|m\|_{L^p(Q_{\tau,\tau+2})}^s}.
		\end{equation}
		Define $\Theta:= \abra{\tau+1 \in \mathbb N\,:\,\|m\|_{L^p(Q_{\tau,\tau + 1})} \leq \|m\|_{L^p(Q_{\tau + 1,\tau+2})}}$. Then for any $\tau\in \Theta$, we have
		\begin{equation*}
			\begin{aligned}
			\text{RHS of (\ref{f10})} &\leq  2\,C_{10}\, \bra{1+ \|m\|_{L^p(Q_{\tau+1,\tau+2})}^{\vartheta}+ \|m\|_{L^p(Q_{\tau+1,\tau+2})}^s}\\
			&\leq C_{11} + \frac 12 \|m\|_{L^p(Q_{\tau+1,\tau+2})},
			\end{aligned}
		\end{equation*}
(where $C_{11}$ depends on the same quantities as $C_9$)
		thanks to the fact that $0<\vartheta, s < 1$.
 From this, \eqref{f10}, and $\psi_\tau\equiv 1$ on $(\tau+1,\tau+2)$, we get
		\begin{equation*}
			\|m\|_{L^p(Q_{\tau+1,\tau+2})} \leq C_{11} + \frac 12 \|m\|_{L^p(Q_{\tau+1,\tau+2})}  \quad \text{ for all } \quad \tau\in \Theta,
		\end{equation*}
		and consequently
		\begin{equation*}
			\|m\|_{L^p(Q_{\tau+1,\tau+2})} \leq 2\,  C_{11} \quad \text{ for all } \quad \tau\in \Theta.
		\end{equation*}
		Thanks to the definition of $\Theta$,  we finally obtain the desired estimate \eqref{es_uniform_m_Lp}. In a similar manner, we obtain \eqref{eventual-es_uniform_m_Lp}.
	\end{proof}
	We are now ready to show the uniform-in-time boundedness of the solution to \eqref{sys_original}.
\medskip

		The uniform-in-time boundedness of $d$ is already shown. It remains to show it for $m$ and $c$. We start by rewriting  equation \eqref{eq_c_cutoff} as
		\begin{equation*}
			\pa_t(\psi_\tau c) = \eps_0\, \psi_\tau \,\Delta c + (\psi_\tau' - \psi_\tau)c + H , \quad \text{ with } H = \psi_\tau[\delta d + \beta m].
		\end{equation*}
		Thanks to the Lemma \ref{uniform_m_Lp}, we know that for any $1\le p<\infty$, 
 $\sup_{\tau \in \mathbb N}\|H\|_{L^p(Q_{\tau,\tau+2})} \leq \beta\, K_p + \delta\,(2 |\Omega|)^{1/p}$.
 Proceeding as in Lemma \ref{uniform_gradient_c_Lq1}, we see 
 that for any $1\leq p<\infty$, there exists $K_p^*>0$ (depending on the same quantities as $K_p$) such that $\sup_{\tau\in \N}\norm{c}_{L^p(Q_{\tau,\tau+2})} \leq K_p^*$. By maximal regularity, recalling that $(\psi_\tau c)(x,\tau) = 0$, we have
(for some $C>0$ depending on $\Omega$, $p$, $n$, $\var_0$, but not $\tau$)
		\begin{equation*}
			\|\psi_\tau c\|_{W^{2,1}_{p}(Q_{\tau,\tau+2})} \leq C\|(\psi_\tau' - \psi_\tau)c + H\|_{L^p(Q_{\tau,\tau+2})} ,
		\end{equation*}
so that picking $p>\frac{n+2}{2}$, we can use the semigroup properties of the heat kernel (Lemma \ref{embedding})
in order
to get 
		\begin{equation*}
			\|\psi_\tau c\|_{L^\infty(Q_{\tau,\tau+2})} \leq C\|\psi_\tau c\|_{W^{2,1}_p(Q_{\tau,\tau+2})} \leq C_{[\frac{n+2}2] + 1}.
		\end{equation*}
		Due to $\psi_\tau \equiv 1$ on $(\tau+1,\tau+2)$, we finally obtain
		\begin{equation*}
			\|c\|_{L^\infty(Q_{\tau+1,\tau+2})} \leq C_{[\frac{n+2}2] + 1} .
		\end{equation*}
		Moreover, by picking $p>n+2$, we obtain also the uniform estimate for $\na c$, i.e.
 (for some $C>0$ depending on $\Omega$, $p$, $n$, $\var_0$, but not on $\tau$)
		\begin{equation}\label{grad_c}
			\|\na c\|_{L^\infty(Q_{\tau+1,\tau+2})} \leq \|\psi_\tau|\na c|\|_{L^\infty(Q_{\tau,\tau+2})} \leq C\|c\|_{W^{2,1}_p(Q_{\tau,\tau+2})} \leq C_{n + 3}.
		\end{equation}
This concludes the proof of boundedness of $c$ (and $\nabla c$).  
\medskip

		For the uniform boundedness of $m$, we observe that from \eqref{eq_m_cutoff}, one can write the Duhamel formula for all $t\in (\tau,\tau+2)$
		\begin{equation*}
		\psi_\tau m(t) = \int_{\tau}^tS(t-\tau-s)\sbra{(\psi_\tau'm) + \psi_\tau m(1-m^{a-1}) - \chi\psi_\tau \na\cdot(f(m)\na c)} (s)\, ds
		\end{equation*}
		where $\{S(t)\}$ is the semigroup of the heat kernel with homogeneous Neumann boundary condition. We have the following estimates of $\{S(t)\}$, (cf \cite {winkler2010aggregation}), for $p$ and $q$ well chosen,
		\begin{equation*}
			\|S(t)f\|_{L^\infty(\Omega)} \leq C^*\bra{1+t^{-\frac{n}{2p}}}\|f\|_{L^p(\Omega)} \; \text{ and } \; \|S(t)\na f\|_{L^\infty(\Omega)} \leq C^*\bra{1+t^{-\frac 12 - \frac{n}{2q}}}\|f\|_{L^q(\Omega)},
		\end{equation*}
where $C^*$ only depends on $\Omega$, $n$, $p$, and $q$.
\par
		Using these estimates into the above Duhamel formula, we get
		\begin{equation*}
		\begin{aligned}
			\sup_{t\in(\tau,\tau+2)}&\|\psi_\tau m(t)\|_{L^\infty(\Omega)}\\
			&\leq 3 \int_\tau^{\tau+2}\|[S(t-\tau-s)m](s)\|_{L^\infty(\Omega)}ds  + \int_{\tau}^{\tau+2}\|[S(t-\tau-s)m^a](s)\|_{L^\infty(\Omega)}ds\\
			&\quad + \chi \int_{\tau}^{\tau+2}\|[S(t-\tau-s)\na \cdot(f(m)\na c)](s)\|_{L^\infty(\Omega)}ds\\
			&\leq 3C^*\, \|m\|_{L^p(Q_{\tau,\tau+2})}\int_{\tau}^{\tau+2}\bra{1+(t-\tau-s)^{-\frac{n}{2p}}}ds\\
			&\quad + C^*\,\|m\|_{L^{ap}(Q_{\tau,\tau+2})}^a\int_\tau^{\tau+2}\bra{1+(t-\tau-s)^{-\frac{n}{2p}}}ds\\
			&\quad + C^*\, \gamma\, \chi\,\|\na c\|_{L^\infty(Q_{\tau,\tau+2})}\|m\|_{L^{bp}(Q_{\tau,\tau+2})}^{b}
\int_{\tau}^{\tau+2}\bra{1+(t-\tau-s)^{- \frac12 -\frac{n}{2p}}}ds.
		\end{aligned}
		\end{equation*}
		By applying Lemma \ref{uniform_m_Lp} and \eqref{grad_c}, we can choose $p>n$, which implies that the last three integrals converge and are  bounded w.r.t $\tau$. We finally obtain the estimate (for a constant $C>0$ depending on the parameters 
stated in Theorem \ref{thm:main})
		\begin{equation*}
			\sup_{t\in(\tau,\tau+2)}\|\psi_\tau m\|_{L^\infty(\Omega)} \leq C.
		\end{equation*}
		That is, \eqref{uniform-in-time} is proved. The result in \eqref{eventual-uniform-in-time} is obtained in a similar manner, by applying the earlier $\limsup$ estimates and parabolic regularity, concluding the proof of Theorem \ref{thm:main}.

\section{Nonlinear stability and Turing instability}\label{sec:stability}

 We study in this section the homogeneous equilibria of system (\ref{sys_original}). Those equilibria $(\bar{m}, \bar{c}, \bar{d})$ satisfy $\bar{m}= {\bar{m}}^a$, $\delta \, \bar{d}  + \beta\, \bar{m} = \bar{c}$, and $g(\bar{m})\, (1- \bar{d})=0$. They are therefore given by $(1, \beta + \delta, 1)$, $(0, \delta, 1)$, and, if $g(0)=0$, $(0, \delta\,\zeta, \zeta)$ for all $\zeta \ge 0$. 
\medskip

It is easy to see that the equilibria $(0, \delta, 1)$, and (if $g(0)=0$)  $(0, \delta\,\zeta, \zeta)$ for all $\zeta \ge 0$, are linearly unstable by considering $x$-independent solutions to (\ref{sys_original}).
\medskip

In this section, we assume that $f(1)>0$. 

\medskip

In order to study the stability of the positive equilibrium $(\bar{m}, \bar{c}, \bar{d}) = (1, \beta + \delta, 1)$, we consider $\tilde{m} := m -  \bar{m}$, $\tilde{c} := c - \bar{c}$,
 $\tilde{d} := d - \bar{d}$. 
The system (\ref{sys_original}) can be rewritten under the form
\begin{equation}\label{xm}
  \pa_t \tilde{X} = L_1\, \tilde{X}  +  L_2\, \tilde{X} + R(\tilde{X}), 
\end{equation}
where
$$ \tilde{X} := \left( \begin{array}{c} \td{m}\\ \td{c}\\ \td{d} \end{array} \right) . $$
The linear part of the system is associated to the matrices
$$  L_1 := \left( \begin{array}{ccc} 1 - a & 0 & 0\\  \beta & -1 & \delta \\ 0 & 0 & - g(1) \end{array} \right),  \qquad L_2 := \left( \begin{array}{ccc} \Delta &  -\chi\,f(1)\, \Delta & 0\\  0 & \var_0\, \Delta & 0 \\ 0 & 0 & 0 \end{array} \right),$$
and the nonlinear part of the system is
$$
R(\tilde{X}) :=  \left(\begin{array}{c}
R_1(\td X)\\ R_2(\td X) \\ R_3(\td X)
\end{array} \right) = 
\left( \begin{array}{c}  -[(1+\tilde{m})^a - 1 - a\,\tilde{m}]
- \chi\nabla\cdot\{ [f(1+\tilde{m}) - f(1)]\, \nabla \tilde{c}  \}\\ 0   \\ - [g(1 + \tilde{m}) - g(1)]\, \tilde{d}\end{array} \right).
$$

\medskip
We will show that when $\chi$ is small enough, the steady state $(\bar{m}, \bar{c}, \bar{d})$ is nonlinearly
asymptotically stable. This is obtained by first proving that the linear part possesses a spectral gap. This implies the exponential decay (for the linear part) of the solution. Thanks to the fact that the solutions are uniformly bounded in time, we then show
that the nonlinear part is dominated by the linear one provided that the initial data are close enough to $(\bar{m}, \bar{c}, \bar{d})$, and thus finally obtain the nonlinear stability.

\medskip

Denoting by $(\lam_n)_{n \in \N}$ the eigenvalues of $-\Delta$ (with Neumann boundary condition) on $\Omega$ in such a way that 
$\lam_0 = 0 < \lam_1 \le \lam_2 \le ... \to +\infty$, we see that the matrix of the linearized system associated to  (\ref{xm}) projected on the eigenspace $\mathbb P\, e_n$ of $L^2(\Omega)$ associated to $\lam_n$, is given by 
$$ M_n := \left( \begin{array}{ccc} 1 - a - \lam_n&  \chi\,f(1)\,\lam_n & 0\\  \beta & -1 - \var_0\,\lam_n& \delta \\ 0 & 0 & - g(1) \end{array} \right) . $$
Recalling that $g(1) >0$, we see that the spectral properties of $M_n$ can be studied by computing the trace and the determinant of the extractred (from $M_n$)  matrix 
$$ N_n := \left( \begin{array}{cc} 1 - a - \lam_n&  \chi\,f(1)\,\lam_n \\  \beta & -1 - \var_0\,\lam_n \end{array} \right) . $$
 The trace of $ N_n$ is 
$$ {\hbox{ Tr }}  N_n = - a - (1 + \var_0)\, \lam_n <0 \quad \text{ for all } \quad n\in \mathbb N. $$
 The determinant of $ N_n $ is 
$$ {\hbox{ Det }}  N_n = [a - 1] + [1 + \var_0\,(a-1) - \chi\,f(1)\,\beta]\,\lam_n + \var_0\, \lam_n^2 . $$ 
We define 
\begin{equation}\label{chi_c}
\chi_c := \sup \{\chi \, | \, \forall n \in \N,  \quad [a - 1] + [1 + \var_0\,(a-1) - \chi\,f(1)\,\beta] \,\lam_n + \var_0\, \lam_n^2 >0\} , 
\end{equation}
and observe that 
\begin{equation}\label{chi_c0}
\chi_c \ge \chi_{c0} := \frac{2 \sqrt{\var_0(a-1)} + 1 + \var_0\,(a-1)}{f(1)\, \beta} > 0.
\end{equation}
Note that $\chi_{c0}$ does not depend 
on $\Omega$ (and there exists $\Omega$ such that $\chi_c = \chi_{c0}$).
\medskip

\subsection{Nonlinear stability}\label{local_stabil}
\begin{proof}[Proof of Theorem \ref{thm:main2}]
Thanks to Theorem \ref{thm:main}, we recall that there exist $K_m, K_c, K_d>0$ such that for all $t>0$,
\begin{equation}\label{Linfbound}
	\norm{\td m(t)}_{L^\infty(\Omega)} \leq K_m, \quad \norm{\td c(t)}_{W^{1,\infty}(\Omega)} \leq K_c, \quad \norm{\td d(t)}_{L^\infty(\Omega)} \leq K_d.
\end{equation}

It follows from the condition \eqref{sub_critical} that 
there exist $\alpha\in (0,1)$ and $\vt_1 > 0$ such that
\begin{equation}\label{a1}
	\vt_1 < \frac{4\alpha(a-1)}{\beta^2} \quad \text{ and } \quad \chi < \frac{2\sqrt{\vt_1\eps_0}}{f(1)}.
\end{equation}
Now we choose
\begin{equation}\label{a2}
 \vt_2 > \frac{\delta^2 \vt_1}{4g(1)(1-\alpha)},
\end{equation}
and define a function
\begin{equation*}
	\varphi(t) = \frac 12\bra{ \|\tilde{m}\|_{L^2(\Omega)}^2 + \vt_1\|\tilde{c}\|_{L^2(\Omega)}^2 + \vt_2\|\tilde{d}\|_{L^2(\Omega)}^2}.
\end{equation*}
Direct calculations give
\begin{equation}\label{phi'}
\begin{aligned}
	\varphi'(t) & = -\norm{\na \td{m}}_{L^2(\Omega)}^2 + \chi f(1)\inner{\na \td m}{\na \td c} - \vt_1\eps_0\norm{\na \td c}_{L^2(\Omega)}^2\\
	& \; + (1-a)\norm{\td m}_{L^2(\Omega)}^2 - \vt_1\norm{\td c}_{L^2(\Omega)}^2 - \vt_2 g(1)\|\td d\|_{L^2(\Omega)}^2 + \beta \vt_1\inner{\td m}{\td c} + \vt_1\delta \inner{\td d}{\td c}\\
	& \; + \inner{R_1(\td X)}{\td m} + \inner{R_3(\td X)}{\vt_2\td d}.
\end{aligned}
\end{equation}
From the relation between $\chi$ and $\vt_1$ in \eqref{a1}, there exists $\omega_1>0$ such that
\begin{equation*}
	-\norm{\na \td{m}}_{L^2(\Omega)}^2 + \chi f(1)\inner{\na \td m}{\na \td c} - \vt_1\eps_0\norm{\na \td c}_{L^2(\Omega)}^2 \leq -\omega_1\bra{\norm{\na \td m}_{L^2(\Omega)}^2 + \norm{\na \td c}_{L^2(\Omega)}^2}.
\end{equation*}
It also follows from the condition of $\vt_1$ in \eqref{a1} that there exists $\omega_2>0$ satisfying
\begin{equation*}
	(1-a)\norm{\td m}_{L^2(\Omega)}^2 + \beta \vt_1 \inner{\td m}{\td c} - \alpha \vt_1\norm{\td c}_{L^2(\Omega)}^2 \leq -\omega_2\bra{\norm{\td m}_{L^2(\Omega)}^2 + \norm{\td c}_{L^2(\Omega)}^2} ,
\end{equation*}
where $\alpha$ is chosen in \eqref{a1}. Finally, from \eqref{a2}, there exists $\omega_3>0$ such that
\begin{equation*}
	-(1-\alpha)\vt_1\norm{\td c}_{L^2(\Omega)}^2 + \vt_1\delta \inner{\td d}{\td c} - \vt_2g(1)\|\td d\|_{L^2(\Omega)}^2 \leq -\omega_3\bra{\|\td c\|_{L^2(\Omega)}^2 + \|\td d\|_{L^2(\Omega)}^2}.
\end{equation*}
Note that these three estimates mean that the linear part of \eqref{xm}, after the rescaling by $\mathrm{diag}(1,\vt_1, \vt_2)$, has a negative spectral gap. Inserting all these estimates into \eqref{phi'}, we obtain for some $\omega_4,\, \omega_5 > 0$ that
\begin{equation}\label{phi2}
	\varphi'(t) \leq -\omega_4\bra{\norm{\td m}_{H^1(\Omega)}^2 + \norm{\td c}_{H^1(\Omega)}^2} - \omega_5\varphi(t) + \inner{R_1(\td X)}{\td m} + \inner{R_3(\td X)}{\vt_2 \td d}.
\end{equation}

To deal with the two last terms on the right-hand side of \eqref{phi2}, we first use integration by parts to have
\begin{equation}\label{h3}
	\begin{aligned}
	&\inner{R_1(\td X)}{\td m} + \inner{R_3(\td X)}{\vt_2 \td d}\\
	&= -\int_{\Omega}\td{m}\sbra{(1+\td m)^a - 1 - a\td m}dx + \chi \int_{\Omega}\sbra{f(1+\td m) - f(1)}\na \td c\cdot \na \td m dx\\
	&\quad - \vt_2\int_{\Omega}\sbra{g(1+\td m)-g(1)}\td{d}^2dx =: (A) + (B) + (C).
	\end{aligned}
\end{equation}

Now we fix  a constant $\kappa = \kappa(n) \in (0,\min\{a-1,1\})$ small enough such that $H^1(\Omega)\hookrightarrow L^{2+2\kappa}(\Omega)$.  Thanks to this embedding and interpolation inequality, we have (for some $C(\Omega,\kappa)$ which only depends 
on $\Omega$ and $\kappa$)
\begin{equation}\label{interpolation}
	\norm{\td m}_{L^{2+\kappa}(\Omega)}^{2+\kappa} \leq \norm{\td m}_{L^{2+2\kappa}(\Omega)}^2\norm{\td m}_{L^{1+\kappa}(\Omega)}^{\kappa} \leq C(\Omega,\kappa)\norm{\td m}_{H^1(\Omega)}^2\norm{\td m}_{L^{2}(\Omega)}^{\kappa}.
\end{equation}

To estimate $(A)$, we use Taylor's expansion and the fact that $\tilde m$ is uniformly bounded in $L^{\infty}(\Omega)$ to get 
\begin{equation}\label{A}
\begin{aligned}
|(A)| &\leq\int_{\Omega}|\tilde{m}| |(1+\tilde{m})^a - 1 - a\,\tilde{m}|dx\\
&\leq C(a, K_m)\int_{\Omega}|\td m|^{1+a}dx\\
&\leq C(a,K_m)K_m^{a-1-\kappa}\int_{\Omega}|\td m|^{2+\kappa}dx \quad (\text{since }\kappa < a-1)\\
&\leq C(a,K_m,\kappa)\norm{\td m}_{L^{2+\kappa}(\Omega)}^{2+\kappa}\\
&\leq C(a,K_m,\kappa)\norm{\td m}_{H^1(\Omega)}^2\norm{\td m}_{L^2(\Omega)}^{\kappa} \quad (\text{using } (\ref{interpolation})).
\end{aligned}
\end{equation}
Similarly for $(B)$ and $(C)$
\begin{equation}\label{C}
\begin{aligned}
|(C)| &\leq \vt_2\int_{\Omega} |g(1 +\tilde{m}) - g(1)|\tilde{d}^2dx \le C(K_m) \int_{\Omega}|\tilde{m}| |\tilde{d}|^2dx \quad (\text{since }g\in C^{1,0})\\
&\leq C(K_m, K_d)\int_{\Omega}|\td m||\td d|^{\frac{2(1+\kappa)}{2+\kappa}}dx\\
&\leq C(K_m,K_d,\omega_5)\int_{\Omega}|\td m|^{2+\kappa}dx + \frac{\omega_5\vt_2}{4}\int_{\Omega}|\td d|^2dx\quad (\text{by Young's inequality})\\
&\leq C(K_m,K_d,\omega_5,\kappa)\norm{\td m}_{H^1(\Omega)}^2\norm{\td m}_{L^2(\Omega)}^{\kappa} + \frac{\omega_5}{2}\varphi(t) \quad (\text{using } (\ref{interpolation})),
\end{aligned}
\end{equation}
and
\begin{equation}\label{B}
\begin{aligned}
|(B)| &\leq \chi \int_{\Omega}|f(1+\tilde{m}) - f(1)||\nabla \tilde{c}||\nabla \tilde{m}|dx\\
&\le C(K_m) \int_{\Omega}|\tilde{m}| 
 |\nabla \tilde{c}||\nabla \tilde{m}|dx \quad (\text{since }f\in C^{1,0})\\
 &\leq C(K_m, K_c)\int_{\Omega}|\td m||\na \td c|^{\frac 13}|\na \td m|dx \quad (\text{thanks to } (\ref{Linfbound}))\\
 &\leq \frac{\omega_4}{2}\int_{\Omega}\bra{|\na \td m|^2 + |\na \td c|^2}dx + C(K_m,K_c,\omega_4)\int_{\Omega}|\td m|^3dx \quad (\text{by Young's inequality})\\
 &\leq \frac{\omega_4}{2}\bra{\norm{\td m}_{H^1(\Omega)}^2 + \norm{\td c}_{H^1(\Omega)}^2} + C(K_m,K_c,\omega_4)K_m^{1-\kappa}\int_{\Omega}|\td m|^{2+\kappa}dx\\
 &\leq \frac{\omega_4}{2}\bra{\norm{\td m}_{H^1(\Omega)}^2 + \norm{\td c}_{H^1(\Omega)}^2} + C(K_m,K_c,\omega_4,\kappa)\norm{\td m}_{H^1(\Omega)}^2\norm{\td m}_{L^2(\Omega)}^{\kappa} \; (\text{using } (\ref{interpolation})).
\end{aligned}
\end{equation}

By inserting \eqref{A}--\eqref{C}--\eqref{B} into \eqref{phi2} we obtain
\begin{align*}
\varphi'(t) &\leq - \frac{\omega_5}{2}\varphi(t) + \bra{\|\td m\|_{H^1(\Omega)}^2 + \|\td c\|_{H^1(\Omega)}^2}\sbra{C\|\td m\|_{L^2(\Omega)}^{\kappa} - \frac{\omega_4}{2}}\\
&\leq - \frac{\omega_5}{2}\varphi(t) + \bra{\|\td m\|_{H^1(\Omega)}^2 + \|\td c\|_{H^1(\Omega)}^2}\sbra{C\varphi(t)^{\kappa/2} - \frac{\omega_4}{2}} .
\end{align*}
Therefore, if $\varphi(0)$ is small enough, $t \mapsto \varphi(t)$ is decreasing and consequently
\begin{equation*}
	\varphi(t) \leq e^{-\frac{\omega_5}{2}t}\varphi(0)
\end{equation*}
which proves the nonlinear (exponential) stability in $L^2$ of $(\bar{m}, \bar{c}, \bar{d})$. Interpolating with the  estimate (\ref{uniform-in-time}), we
conclude the proof of Thm. \ref{thm:main2}.
\end{proof}

\subsection{Turing instability and Turing patterns}\label{subsec:Turing}
When $\chi > \chi_c$, the equilibrium point $(1, \beta + \delta, 1)$ is linearly unstable, so that a Turing-type instability sets in. In this section, we display a numerical simulation  showing the occurrence of this instability on the spatial domain $ \Omega=[0;L_x]\times[0;L_y] $. 

First, we recall that on rectangular domains defined by $ 0<x<L_x $ and $ 0<y<L_y $, the Fourier transformation of solutions to the linearized problem with homogeneous Neumann boundary conditions has the following form:

\begin{equation}\label{eq:fourier}
\tilde{X}=\sum_{p,q \in \N } \textbf{f}_{p,q} e^{\sigma(\lambda_n) t} \cos\left(\frac{p \pi}{L_x}x\right)\cos\left(\frac{q \pi}{L_y}y\right),
\end{equation}
where $ \textbf{f}_{p,q} $ are the Fourier coefficients of the initial conditions, and $ \lambda_n $ are defined as follows:
\begin{equation}\label{eq:lambda_cond}
\lambda_n=\left(\frac{p \pi}{L_x}\right)^2+\left(\frac{q \pi}{L_y}\right)^2, \text{ with } p, \, q \in \N,
\end{equation}
in such a way that they are ordered (that is $  \lambda_n \leq \lambda_{n+1}$). The values $ \sigma(\lambda_n) $ are derived from the dispersion relations $ \sigma^2 - \mathrm{Tr} N_n \sigma + \mathrm{Det} N_n = 0 $. When $ \chi<\chi_c, \, \sigma(\lambda_n)<0, \, \forall n \in \N $, and then the homogeneous state is linearly stable.
For $ \chi=\chi_c $, the homogeneous solution is marginally stable and $ \exists \lambda_{n_{c}} $ such that $ \min(\mathrm{Det} N_{n_{c}})=0 $, or in other words $ \sigma(\lambda_{n_{c}})=0 $. For $ \chi>\chi_c $, there exists a range $ [n_1; n_2] $ such that $ \sigma(\lambda_n) > 0$, $ \forall n \in [n_1; n_2] $,
 and therefore a Turing instability occurs.

The numerical simulation has been performed on the square domain $ [0;\pi]\times[0;\pi] $, i.e. $L_x = L_y = \pi$. For the discretization in space, we have adopted a
Fourier spectral scheme with $ 64 \times 64 $ modes. The integration in time has been realized by using the Crank-Nicholson method for the diffusive part and a second-order Runge-Kutta explicit method for
the reaction terms. 

We have adopted the functions
\begin{equation*}
	f(m) = \frac{m}{1+m} \quad \text{ and } \quad g(m) = \frac{rm^2}{1+m},
\end{equation*}
which corresponds to the special case \eqref{special} (except for parameter $a$), and the following parameter values have been fixed: $ r=1, \, a=3, \,\varepsilon_0=0.03125, \, \delta=1, \, \beta=1 $.
For this set of values, linear stability analysis predicts that $ \lambda_{n_{c}}=8 $ and $ \chi_c=3.125 $. In order to obtain a Turing pattern, we have fixed $ \chi=3.18 $ (which is bigger than $\chi_c$) and for this choice the only unstable $ \lambda_n $ allowed by the domain and the boundary conditions is $ \lambda_{n_{c}} $, which has multiplicity one, in that $ (p, q)=(2,2) $ is the only pair which satisfies the condition \eqref{eq:lambda_cond}. 

We therefore expect that a random perturbation of the homogeneous state evolves towards a
 stationary Turing pattern of the form $ \tilde{X} \propto \cos(2 x)\cos(2 y) $.

The temporal evolution of the numerical solution is shown in Fig.\ref{fig:Turing_instability}. As predicted by the previous analysis, a small random perturbation (see Fig.\ref{fig:Turing_a}) of the homogeneous state evolves toward a stationary square patten (Fig.\ref{fig:Turing_d}). 

\begin{figure}[ht]
	\centering
	\subfloat[]
	{\label{fig:Turing_a}\includegraphics[width=.4\textwidth]{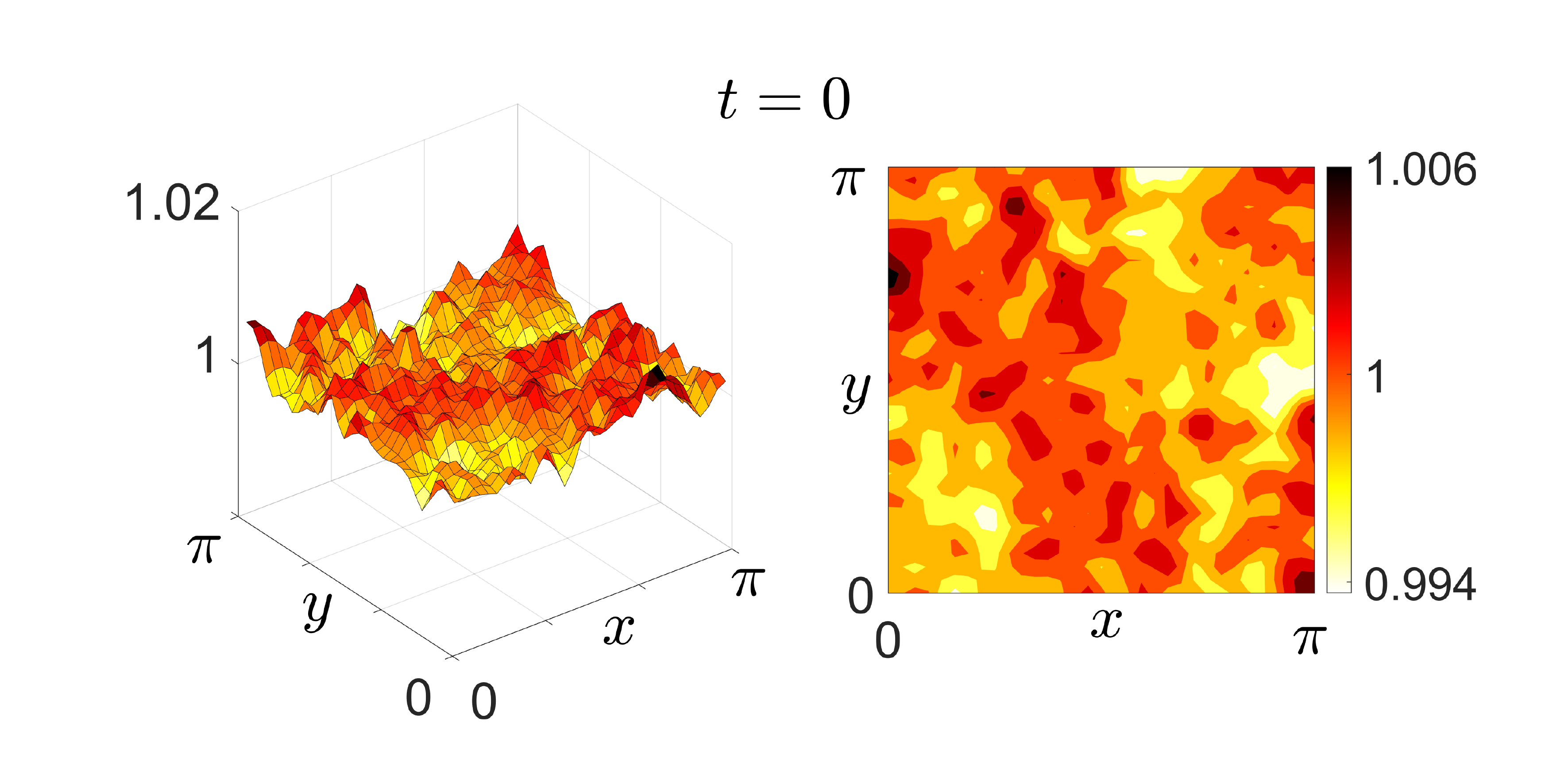}}
	\hspace{1mm}
	\subfloat[]
	{\label{fig:Turing_b}\includegraphics[width=.4\textwidth]{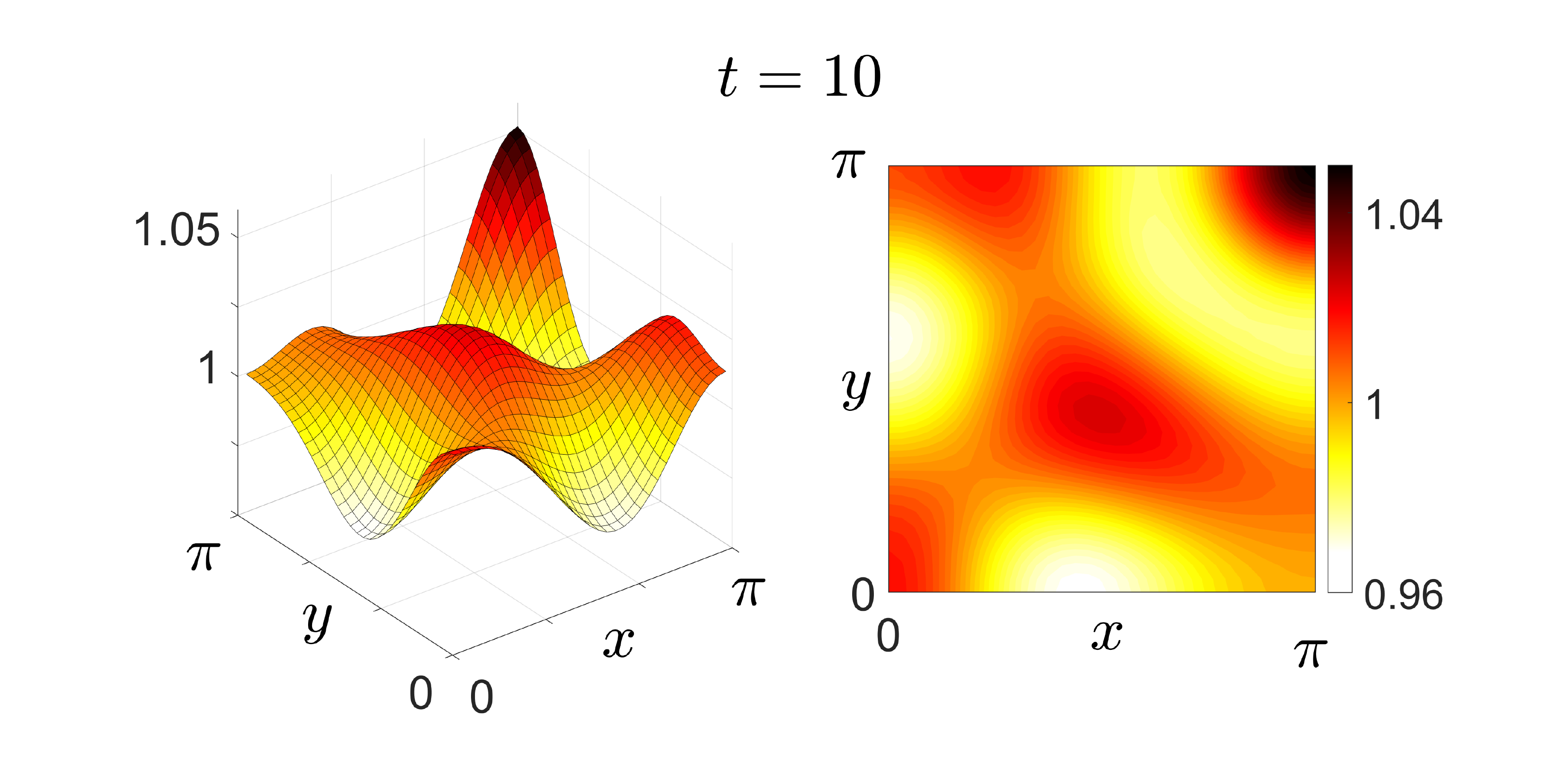}}\\
	\subfloat[]
	{\label{fig:Turing_c}\includegraphics[width=.4\textwidth]{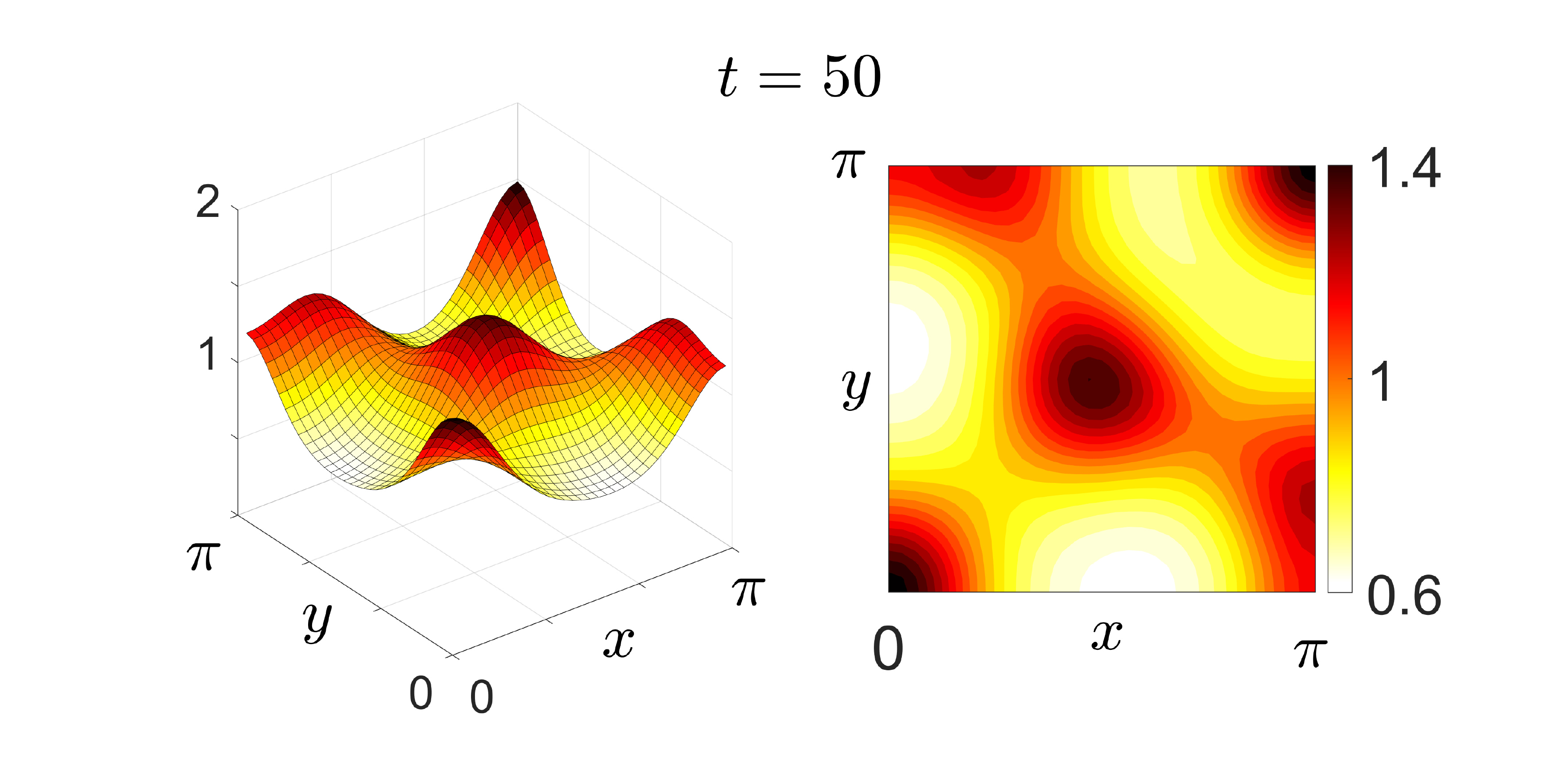}}
	\hspace{1mm}
	\subfloat[]
	{\label{fig:Turing_d}\includegraphics[width=.4\textwidth]{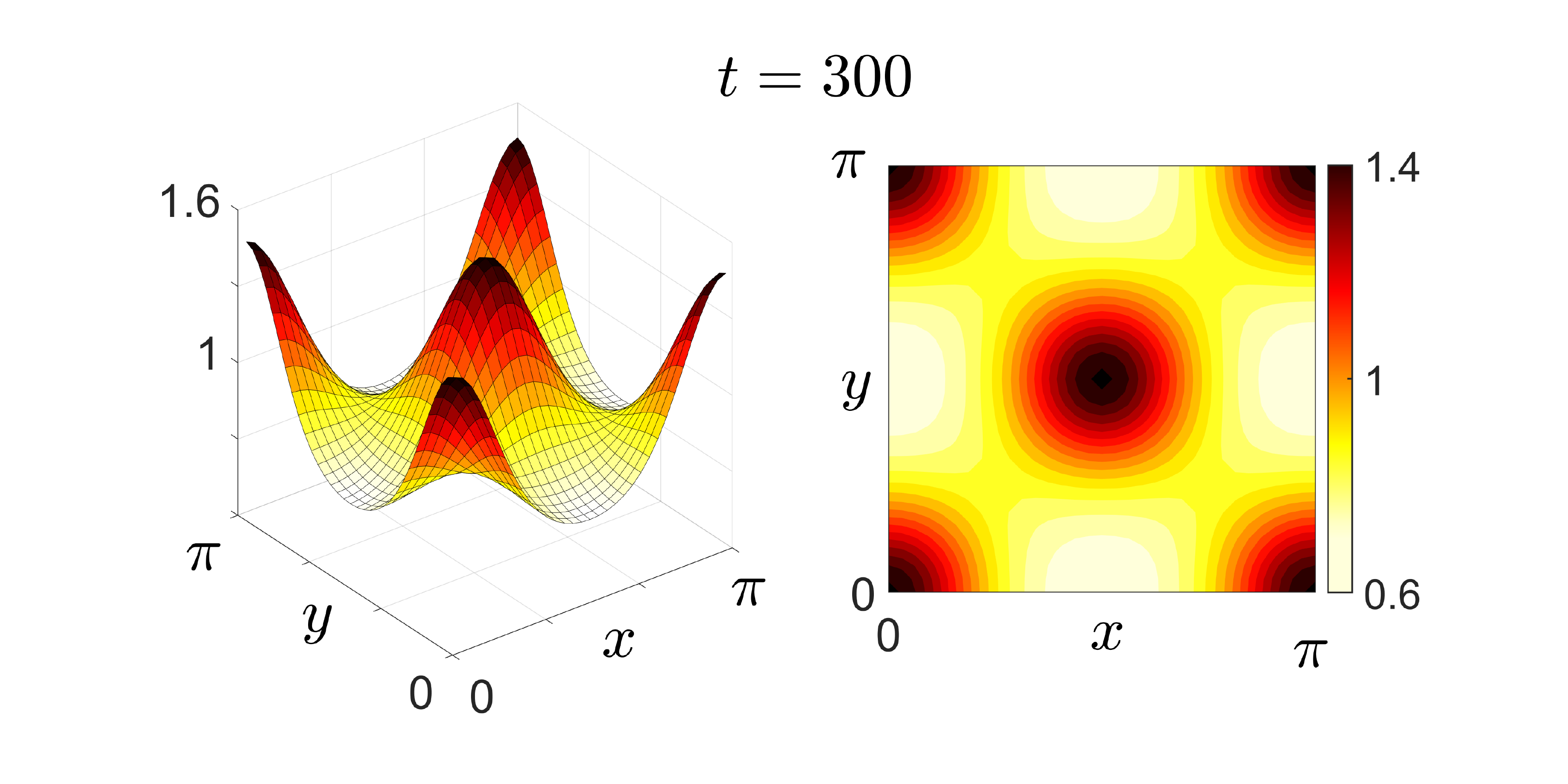}}
	\caption{Spatio-temporal evolution of $ m $ on the spatial domain $\Omega= [0;\pi]\times[0;\pi] $ above the Turing threshold $ \chi_c=3.125 $. The following parameter values have been fixed: $ r=1, \, a=3, \,\varepsilon_0=0.03125, \, \delta=1, \, \beta=1 $ and $ \chi=3.18. $ The initial condition is a small random perturbation of the homogeneous equilibrium $\bar m = 1$.}
	\label{fig:Turing_instability}
\end{figure}
On the other hand, when $ \chi<\chi_c $, the homogeneous steady state is stable and therefore perturbations of this equilibrium, as shown in Fig.\ref{fig:Turing_stable}, are dampened.
\begin{figure}[ht]
\centering
\includegraphics[width=.8\textwidth]{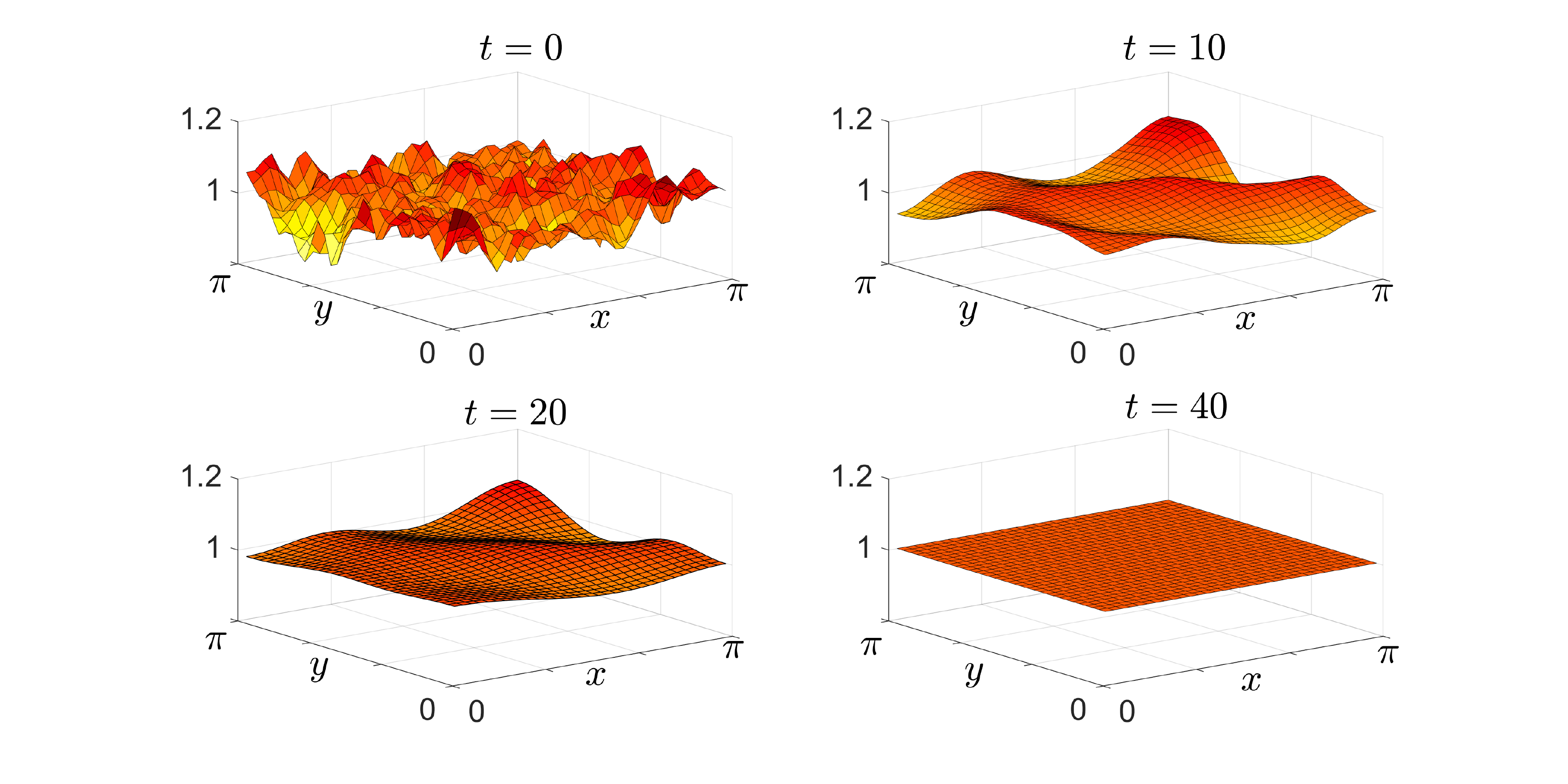}
\caption{Spatio-temporal evolution of $ m $ on the spatial domain $\Omega= [0;\pi]\times[0;\pi] $ below the Turing threshold $ \chi_c $. All parameter values, except $\chi$, are as in Fig.\ref{fig:Turing_instability}, and
 $ \chi=3 $. The initial condition is a random perturbation of the homogeneous equilibrium $\bar m = 1$.}
\label{fig:Turing_stable}
\end{figure}
\subsection{Global asymptotic stability}\label{subsec:GlobalStab}
We expect that the convergence to the steady state $(\bar m, \bar c, \bar d) = (1,\beta + \delta, 1)$ is not only local, as proved in Subsection \ref{local_stabil} when $\chi < \chi_{\mathrm{subcrit}}$, but is in fact global and valid for $\chi < \chi_c$. This is left as an interesting open question. We performed numerical simulations which support this conjecture. We show one of these simulations in Fig.\ref{fig:Global_asympt_stability}: below the Turing threshold $ \chi_c $, an initial datum far from the homogeneous equilibrium evolves toward the stable equilibrium point $ (\bar m, \bar c, \bar d) $.
\\Moreover we performed several other simulations adopting initial data further away from equilibrium and in any of these numerical experiments we observed solutions converging to the stable steady state.
\par 
Note that the large time behavior of system (\ref{sys_original}) when $\chi > \chi_c$ is a much more complicated problem: it involves the research
of inhomogeneous stable steady states, and the study of the basin of attraction.
\begin{figure}[ht]
	\centering
	\includegraphics[width=.8\textwidth]{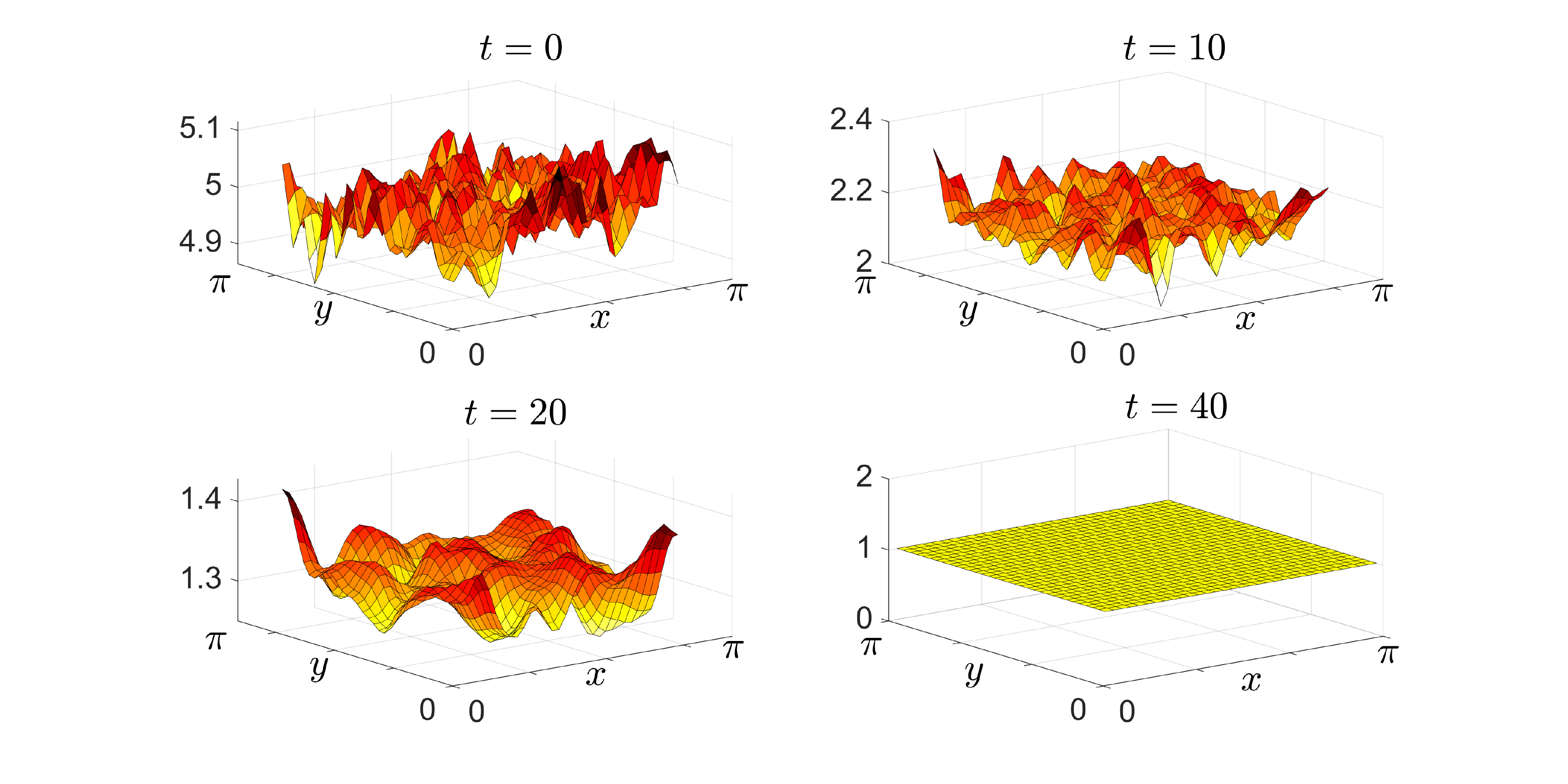}
	\caption{Spatio-temporal evolution of $ m $ on the spatial domain $\Omega= [0;\pi]\times[0;\pi] $ below the Turing threshold $ \chi_c $. All parameter values are as in Fig.\ref{fig:Turing_stable}. The initial condition is far from the homogeneous equilibrium $(\bar m, \bar c, \bar d)$.}
	\label{fig:Global_asympt_stability}
\end{figure}

\medskip
\noindent{\bf Acknowledgement.} This work is partially supported by NAWI Graz, and the International Research Training Group IGDK 1754 ``Optimization and Numerical Analysis for Partial Differential Equations with Nonsmooth Structures'', funded by the German Research Council (DFG) project number 188264188/GRK1754 and the Austrian Science Fund (FWF) under grant number W 1244-N18.


\begin{thebibliography}{00} 

\bibitem[BGGLPS18]{BGGLPS18} E. Bilotta, F. Gargano,  V. Giunta, M. C. Lombardo, P. Pantano, M. Sammartino.
\newblock Eckhaus and zigzag instability in a chemotaxis model of multiple sclerosis. 
\newblock {\em Atti della Accademia Peloritana dei Pericolanti-Classe di Scienze Fisiche, Matematiche e Naturali}, \textbf{96}(S3) (2018) 9.

\bibitem[BGGLPS19]{BGGLPS19} E. Bilotta, F. Gargano,  V. Giunta, M. C. Lombardo, P. Pantano, M. Sammartino.
\newblock Axisymmetric solutions for a chemotaxis model of Multiple Sclerosis. 
\newblock {\em Ricerche di Matematica}, \textbf{68}(1) (2019) 281--294.

\bibitem[CDF14]{canizo2014improved}
Jos{\'e}~A Ca{\~n}izo, Laurent Desvillettes, and Klemens Fellner.
\newblock Improved duality estimates and applications to reaction-diffusion
  equations.
\newblock {\em Communications in Partial Differential Equations},
  39(6):1185--1204, 2014.

\bibitem[CK08]{calvez2008mathematical}
Vincent Calvez and Roman~H Khonsari.
\newblock Mathematical description of concentric demyelination in the human
  brain: Self-organization models, from liesegang rings to chemotaxis.
\newblock {\em Mathematical and computer modelling}, 47(7-8):726--742, 2008.

\bibitem[DG20]{desvillettes2020existence}
Laurent Desvillettes and Valeria Giunta.
\newblock Existence and regularity for a chemotaxis model involved in the
  modeling of multiple sclerosis.
\newblock {\em Ricerche di Matematica}, pages 1--15, 2020.

\bibitem[DT15]{DTres}



Laurent Desvillettes and Ariane Trescases.
\newblock New results for triangular reaction cross diffusion systems.
 \newblock {\em  Journal of Mathematical Analysis and Applications}, 430, n.1 (2015), 32-59.


\bibitem[DLMT15]{DLMT}
Laurent Desvillettes, Thomas Lepoutre, Ayman Moussa and Ariane Trescases.
\newblock On the entropic structure of reaction-cross diffusion systems.
\newblock {\em Communications in Partial Differential Equations},  40, n.9 (2015), 1705-1747.



\bibitem[GT15]{gilbarg2015elliptic}
David Gilbarg and Neil~S Trudinger.
\newblock {\em Elliptic partial differential equations of second order}.
\newblock springer, 2015.

\bibitem[HFA20]{hu2020global}
Xiaoli Hu, Shengmao Fu, and Shangbing Ai.
\newblock Global asymptotic behavior of solutions for a parabolic-parabolic-ode
  chemotaxis system modeling multiple sclerosis.
\newblock {\em Journal of Differential Equations}, 269(9):6875--6898, 2020.

\bibitem[KC07]{khonsari2007origins}
Roman~H Khonsari and Vincent Calvez.
\newblock The origins of concentric demyelination: self-organization in the
  human brain.
\newblock {\em PLoS One}, 2(1):e150, 2007.

\bibitem[LBBGPS17]{lombardo2017demyelination}
MC~Lombardo, R~Barresi, E~Bilotta, F~Gargano, P~Pantano, and Mml Sammartino.
\newblock Demyelination patterns in a mathematical model of multiple sclerosis.
\newblock {\em Journal of mathematical biology}, 75(2):373--417, 2017.

\bibitem[LSU88]{ladyvzenskaja1988linear}
Olga~A Lady\v{z}enskaija, Vsevolod~Alekseevich Solonnikov, and Nina~N Uralceva.
\newblock {\em Linear and quasilinear equations of parabolic type}, volume~23.
\newblock American Mathematical Soc., 1988.

\bibitem[MT20]{morgan2020boundedness}
Jeff Morgan and Bao~Quoc Tang.
\newblock Boundedness for reaction--diffusion systems with lyapunov functions
  and intermediate sum conditions.
\newblock {\em Nonlinearity}, 33(7):3105, 2020.

\bibitem[Pie10]{pierre10} Michel Pierre.
\newblock Global existence in reaction-diffusion systems with control of mass: a survey.
\newblock {\em Milan Journal of Mathematics} 78.2 (2010): 417-455.

\bibitem[Sim86]{simon1986compact}
Jacques Simon.
\newblock Compact sets in the space $L^p(0,T;B)$.
\newblock {\em Annali di Matematica pura ed applicata}, 146(1):65--96, 1986.

\bibitem[Win10]{winkler2010aggregation}
Michael Winkler.
\newblock Aggregation vs. global diffusive behavior in the higher-dimensional
  keller--segel model.
\newblock {\em Journal of Differential Equations}, 248(12):2889--2905, 2010.

\end{thebibliography}
\end{document}